\documentclass[a4paper]{amsart}
\usepackage[all]{xy}
\usepackage[colorlinks=true]{hyperref}
\usepackage{enumerate}
\usepackage{amssymb}

\usepackage{mathtools}\usepackage{pstricks-add}

\usepackage{color}
\usepackage{xcolor}
\usepackage{graphicx}

\newtheorem{thm}{Theorem}[section]

\newtheorem{remark}[thm]{Remark}

\newtheorem{prop}[thm]{Proposition}
\newtheorem{lem}[thm]{Lemma}
\newtheorem{cor}[thm]{Corollary}

\newtheorem{convention}[thm]{Convention}

\theoremstyle{definition}
\newtheorem{defn}[thm]{Definition}
\newtheorem{ex}[thm]{Example}

\theoremstyle{remark}

\newcommand{\Rr}{\mathbb R}

\newcommand{\Zz}{\mathbb Z}
\newcommand{\Nn}{\mathbb N}
\newcommand{\Cc}{\mathbb C}

\newcommand{\Kk}{\mathbb K}

\newcommand{\set}[1]{\left\{#1\right\}}
\newcommand{\eval}[1]{\left\langle#1\right\rangle}
\newcommand{\brr}[1]{\left[#1\right]}

\newcommand{\F}{\ensuremath{\mathcal{F}}}

\newcommand{\ds}{\displaystyle}
\newcommand{\E}{\ensuremath{\mathcal{E}}}

\renewcommand{\d}{\mathrm d}               

\newcommand{\smalcirc}{\mbox{\,\tiny{$\circ $}\,}}  

\renewcommand{\top}{\text{\rm top}}     
\newcommand{\al}{\alpha}

\DeclareMathOperator{\ad}{ad}           
\DeclareMathOperator{\End}{End}         
\DeclareMathOperator{\Der}{Der}         
\DeclareMathOperator{\str}{Str}

\begin{document}

\title{Modular class of Lie $ \infty$-algebroids and adjoint representations}

\author{Raquel Caseiro}
\address{University of Coimbra\\ CMUC\\ Department of Mathematics\\
Apartado 3008\\
EC Santa Cruz\\
3001-501 Coimbra\\ Portugal}
\email{raquel@mat.uc.pt}

\author{Camille Laurent-Gengoux}
\address{Institut \'Elie Cartan de Lorraine \\ 
	UMR 7502 du CNRS \\ 
	Université de Lorraine \\
	Metz, France }
\email{camille.laurent-gengoux@univ-lorraine.fr}

\thanks{Raquel Caseiro is partially supported by the Centre for Mathematics of the University of Coimbra - UIDB/00324/2020, funded by the Portuguese Government through FCT/MCTES.
	Camille Laurent-Gengoux thanks the "Granum" project MITI 80 primes CNRS for its financial support while this project was completed.   
}



\begin{abstract}
We study the modular class of $Q$-manifolds, and in particular of negatively graded Lie $\infty$-algebroid. We show the equivalence of several descriptions of those classes, that it matches the  classes introduced by various authors and that the notion is homotopy invariant. In the process, the adjoint and coadjoint actions up to homotopy of a Lie $\infty$-algebroid are spelled out. We also wrote down explicitly some dualities, e.g. between representations up to homotopies of Lie $\infty$-algebroids and their $Q$-manifold equivalent, which we hope to be of use for future reference. 
\end{abstract}

\maketitle

\section*{Introduction}             %
\label{sec:introduction}           %

The modular class can be broadly defined as the obstruction to the existence of an invariant volume form. For Lie algebras, it is a class in the Chevalley-Eilenberg cohomology obstructing the existence of a bi-invariant volume form on the corresponding Lie group. For a regular foliation, it is the obstruction to the existence of a transverse volume form invariant under all monodromies; It is valued in tangent cohomology. For Poisson structures, it is the obstruction to the existence of a volume form preserved under all Hamiltonian vector fields, and is valued in Poisson cohomology \cite{Koszul}. Now, Lie algebras, regular foliations, and Poisson structures are Lie algebroids, an object at the heart of Professor Kirill Mackenzie's work: he wrote two books which are the main references on the matter \cite{Kirill05,Kirill85}, and made numerous crucial contributions to their studies.  It is therefore not surprising that the various and disparate occurrences of modular classes have long been unified as particular cases of modular classes of Lie algebroids \cite{Weinstein97}-\cite{ELW}. To be more precise, it has been rightly suggested by Yvette Kosmann-Schwarzbach and Alan Weinstein that modular classes should be associated not to a Lie algebroid, but to a Lie algebroid morphism: the modular class of a Lie algebroid being the one of its anchor map \cite{YKSAL}-\cite{KLW08}- \cite{CF}-\cite{Caseiro}.
Notice that the aforementioned articles make a crucial use of Kirill Mackenzie's explicit description of Lie algebroid morphisms detailed in Part I, chapter 4 of \cite{Kirill05}, first developed by himself in \cite{Kirill87}.

For finite dimensional Lie algebroids, there is dual point of view, attributed to Vaintrob \cite{Vaintrob}, but used also by Kirill MacKenzie and Ping Xu in their studies of Lie bialgebroids \cite{MX}, which consists in seeing the Lie algebroid as a graded manifold equipped with a degree $+1$ vector field squaring to zero.
 As explained by Voronov \cite{V12} in \emph{Q-manifolds and Mackenzie theory: an overview},  $Q$-manifolds are an efficient manner to deal with Lie bialgebroids.
When they are of finite rank, $Q$-manifolds, i.e. graded manifolds concentrated in negative degrees equipped with a degree $ +1$ self-commuting vector field, can be seen as being the dual of Lie $\infty $-algebroids.
They are a natural tool in gauge theory \cite{Kotov,Salnikov} and higher structures \cite{CamposCalaque}.  Of course, it contains (when the base manifold is a point) Lie $\infty $-algebras, whose role is well-known in various deformation theories, in particular deformation quantizations \cite{Kontsevich} or deformations of Poisson manifolds \cite{Cattaneo}, as well as complex geometry \cite{Kapranov} and, more generally, the Atiyah class of a Lie algebroid pair \cite{ChenXu,LSX}. 
In most cases, the only properties of Lie $\infty $-algebroids   of interest are those which are preserved under an equivalence relation called homotopy equivalence (See \cite{Campos} for natural interpretations): this aspect is not seen for Lie algebroids, for which  homotopy equivalence reduces to Lie algebroid isomorphisms.

For instance, an unique up to homotopy "universal Lie $\infty $-algebroid" has also been associated  to singular foliations in \cite{Fregier,Lavau,LLS} and Lie-Rinehart algebras \cite{Louis}.

In the present article, we describe the modular class of negatively graded Lie $\infty $-algebroids. Those are used by Sylvain Lavau \cite{Lavau22} to define modular class of a singular foliation as being the one of (any one of) its universal Lie $\infty$-algebroid.

The modular class of a Lie $\infty$-algebroid has already been considered (in $\Zz_2$-graded setting) by several authors: A. J. Bruce \cite{AB} and T. Voronov \cite{V07} (notice that the modular class appears in the arXiv version of \cite{V07}  -page 6- but not in the printed version \cite{V12}).  In both works, it is the obstruction to the existence of an Berezinian form invariant with respect to the homological vector field $Q$.  Also and independently, Gran\aa ker \cite{G08} describes (via operads)  unimodular Lie $\infty$-algebras and proves that the notion is  homotopy invariant. 
However, as far as we know, the homotopy invariance of modular  class for  Lie $\infty$-algebroids was never considered,
 nor was this class related with representations up to homotopy, more precisely with the adjoint action up to homotopy. We also give some of its geometric properties when restricted to a leaf.
 Since the final goal shall be the study of the modular class of singular foliations \cite{Lavau}, we intend to insist more on the modular class of a Lie $\infty $-algebroid and on its homotopy invariance, but we start with the Lie $\infty $-algebra case.
 
 
We start in Sections \ref{LieInfty1} and \ref{LieInftyModular2}   with a detailed study of the modular forms and class of a Lie $\infty $-algebra (mainly those of finite dimension, although this condition may be relaxed). We show that it can be defined either as the supertrace of the adjoint action, or as the divergence of the vector field that dualizes the Lie $ \infty$-brackets. We also show that it is well-behaved under Lie $\infty$-morphisms, and their homotopies.  

In Sections \ref{LieInfty} and \ref{LieInftyModular}, we enlarge this construction to negatively graded $Q$-manifolds.  Again, we show that it can be defined either as the divergence of the $Q$-vector field, but also as the super-trace of the adjoint action. 
This requires a precise description of adjoint and coadjoint actions for $Q$-manifolds that extend the Abad-Crainic adjoint representations of to homotopy \cite{GSM,AC,Raj} for Lie algebroids. We then show invariance under homotopy equivalence. Various examples are then given, and the geometric meaning is detailed.

\tableofcontents

\section{Lie ${{\infty}}$-algebras}
\label{LieInfty1}

We begin by reviewing some concepts about graded vector spaces and Lie $\infty$-algebras \cite{LS}. Different authors use different conventions: our conventions match those in \cite{Azimi,Campos,LLS,LS,Ryvkin}.

\subsection{Conventions on graded vector spaces}

We will work  with $\Zz$-graded vector spaces with finite dimension over a field $\Kk=\Rr$ or $\Cc$.

Let $E=\oplus_{i\in\Zz}E_i$ be a {\textbf{finite dimensional graded vector space}} (i.e. all vector spaces $E_i$ are of finite dimension, and this dimension is zero except for finitely many of them). We call $E_i$ the homogeneous component of $E$ of degree $i$. An  element $x$ of $E_i$ is said to be homogeneous with degree $|x|=i$.
For each $k\in\Zz$, one may shift all the degrees by $k$ and obtain a new grading on $E$. This new graded vector space is denoted by $E[k]$ and is defined by $E[k]_i=E_{i+k}$.

A morphism $\Phi:E\to V$ between two graded vector spaces is a degree preserving  linear map, i.e. a collection of linear maps $\Phi_i:E_i\to V_i$, $i\in\Zz$. We call $\Phi:E\to V$ a morphism of degree $k$, for some $k\in\Zz$, if it is a morphism between $E$ and $V[k]$. 

The dual $E^*$ of $E$ is naturally a finite dimensional graded vector space whose component of degree $i$ is the dual $(E_{-i})^* $ of $E_{-i}$, for all $ i \in {\mathbb Z}$, in equations: $(E^*)_i = (E_{-i})^*$.

Given two finite dimensional graded vector spaces $E$ and $V$, their direct sum $E\oplus V$ [resp. tensor product $E\otimes V$] is a finite dimensional graded vector space with grading
$$
(E\oplus V)_i= E_i\oplus V_i \hspace{.5cm}
\hbox{ [resp. $(E\otimes V)_i= \oplus_{j+k=i}E_j\otimes V_k$ ]}.
$$

We  adopt the Koszul sign convention: for homogeneous morphisms $f:E\to V$ and $g:F\to W$, the tensor product $f\otimes g:E\otimes F\to V\otimes W$ is the morphism of degree $|f|+|g|$ given by
\begin{equation*}
(f\otimes g)(x\otimes y)=(-1)^{|x||g|} f(x)\otimes g(y),\quad x\in E, y\in F.
\end{equation*}

For each $k\in\Nn_0$, let $T^k(E)=\otimes^k E$ and let $T(E)=\oplus_{k} T^k(E)$ be  the tensor algebra over $E$. The {\textbf{graded symmetric algebra over $E$}} is the quotient
$$
S(E)=T(E)/\eval{x\otimes y- (-1)^{|x||y|}y\otimes x}.
$$
This quotient is a graded commutative algebra, whose product we denote by $ \odot $.

For  $n\geq 1$, let $S_n$ be the permutation group of order $n$. For any $n$-tuple of homogeneous elements $x=(x_1,\ldots,x_n)$ in $E$ and any $\sigma \in S_n$, the \textbf{Koszul sign} is the element in $\epsilon(\sigma,x) \in  \{-1,1\}$ defined by
$$
x_{\sigma(1)} \odot \ldots  \odot x_{\sigma(n)}=\epsilon(\sigma,x) \, x_1 \odot \ldots \odot x_n.
$$
For the sake of simplicity, we will simply denote the Koszul sign by $\epsilon(\sigma) $ instead of $ \epsilon(\sigma,x)$.

An element $\sigma$ of $S_{n}$ is called an $(i,n-i)$-unshuffle if  $\sigma(1)<\ldots< \sigma(i)$ and $\sigma(i+1)<\ldots < \sigma(n)$.
The set of $(i,n-i)$-unshuffles is denoted by $Sh(i,n-i)$.

Since we consider $E$ a  finite dimensional graded vector space, we identify $S(E^*)$ with $(SE)^*$. 

Koszul sign conventions and degree reasons yield, for each homogeneous elements  $f,g\in E^*$,
\begin{eqnarray*}
(f\odot g)(x\odot y)&=& (f\otimes g)(x\odot y+(-1)^{|x||y|}y\odot x)\\
&=&
(-1)^{|x||g|}f(x) g(y) + (-1)^{|x||y|+|y||g|} f(y) g(x),\\
 &=&
(-1)^{|x||g|}f(x) g(y) + f(y) g(x)\quad x,y\in E.
\end{eqnarray*}


\subsection{Lie $\infty$-algebras}

We will consider the symmetric approach to Lie $\infty$-algebras, as in  \cite{LS,Ryvkin}.

\begin{defn}
A \textbf{symmetric Lie $\infty$-algebra} or a \textbf{Lie$[1]$ $\infty$-algebra} is a graded vector space $E=\oplus_{i\in\Zz} E_i$ together with a family of degree $1$ linear maps $l_k: S^k(E)\to E$, $k\geq 1$, satisfying
\begin{equation}\label{eq:def:symm:L:infty:algebra}
\sum_{i+j=n+1}\sum_{\sigma\in Sh(i,j-1)}\epsilon(\sigma) \, l_j\left(l_i\left(x_{\sigma(1)},\ldots, x_{\sigma(i)}\right),x_{\sigma(i+1)}\ldots, x_{\sigma(n)}\right)=0,
\end{equation}
for all $n\in\Nn$ and all homogeneous elements $x_1,\ldots,x_n\in E$.
\end{defn}

The d\'ecalage isomorphism (see, e.g. \cite{Azimi}) establishes a one to one correspondence between  skew-symmetric Lie $\infty$-algebra structures over $E$ and symmetric Lie $\infty$-algebra structures over $E[1]$.

\begin{ex}[Symmetric graded Lie algebra]
A symmetric graded Lie algebra is a symmetric Lie $\infty$-algebra $E=\oplus_{i\in\Zz}E_i$ such that $l_n = 0$ for $n \neq 2$. Then the degree $0$ bilinear map on $E[1]$ defined by:
 \begin{equation}
 \label{eq:decalage}
  [x,y] := (-1)^{j} l_2(x,y) \hbox{ for all $x \in E_i,y\in E_j $}
  \end{equation}
is a graded Lie bracket.
In particular, if $E=E_{-1}$ is concentrated in degree $-1$, we get a Lie algebra structure.
\end{ex}

\begin{ex}[Symmetric DGL algebra]
A symmetric differential graded Lie algebra (DGLA) is a symmetric  Lie $\infty$-algebra $E=\oplus_{i\in\Zz}E_i$  such that $l_n=0$ for $n \neq 1$ and $n \neq 2$.
Then $d:=l_ 1$ is a degree $+1$ linear map $d:E\to E$  squaring zero and satisfies the following compatibility condition with the bracket $ \set{\cdot,\cdot}= l_2(\cdot,\cdot)$:
\begin{equation}
\label{eq:gradedLie}
\left\{\begin{array}{l}
d\set{x,y} + \set{d(x),y} + (-1)^{|x|}\set{x, d(y)}=0,\\
\set{\set{x,y},z} + (-1)^{|y||z|}\set{\set{x,z},y} + (-1)^{|x|}\set{x,\set{y,z}}=0.
\end{array}\right.
\end{equation}
\end{ex}

\begin{ex}\label{ex:DGLA:End:E}
Let $(E=\oplus_{i\in\Zz}E_i, \d)$ be a cochain complex. Then $\End (E)[1]=(\oplus_{i\in\Zz} \End_i E)[1]$ has a natural symmetric DGL algebra  structure with $l_1=\partial, \;\;l_2=\set{\, , \,}$ given by:
\begin{equation*}
\left\{\begin{array}{l}
 \partial\phi
=-\d \smalcirc \phi + (-1)^{|\phi|+1} \phi\smalcirc \d,\\
 \set{\phi,\psi}
=(-1)^{|\phi|+1}\left(\phi\smalcirc \psi - (-1)^{(|\phi|+1)(|\psi|+1)} \psi\smalcirc\phi \right),
\end{array}\right.
\end{equation*}
for $\phi, \psi$ homogeneous elements of $\End(E)[1]$.
\end{ex}

\noindent
Recall that for $(E, \set{l_k}_{k\geq 1})$, a symmetric Lie $\infty$-algebra, equations (\ref{eq:def:symm:L:infty:algebra}) establish:
\begin{itemize}
	\item[(i)]  for $n=1$, that
	$l_1\smalcirc l_1=0$, so that $l_1:E_\bullet \to E_{\bullet +1}$ is a differential on $E$ and we have an associated cohomology  $H^\bullet(E,l_1)$;
	\item[(ii)]  for $n=2$, that
	$$l_1(l_2(x_1,x_2))+ l_2(l_1(x_1),x_2) + (-1)^{|x_1|}l_2(x_1,l_1(x_2))=0,$$
	so that the bracket $l_2$ induces a graded symmetric bracket on $H^\bullet(E,l_1)$;
	\item[(iii)] for $n=3$, that the previously defined symmetric bracket on  $H^\bullet(E,l_1)$ satisfies the graded symmetric Jacobi identities as in equations (\ref{eq:gradedLie}).
\end{itemize}

Hence:

\begin{prop}\label{prop:cohom}Let $(E, \set{l_k}_{k\geq 1})$ be a symmetric Lie $\infty$-algebra.
The graded vector space  $H^\bullet(E,l_1)$ has a natural graded symmetric Lie algebra structure.
\end{prop}


\subsection{Symmetric Lie ${{\infty}}$-algebras as graded manifolds over a point}


\label{sec:arity}

Let $(E,  \set{l_k}_{k\geq 1})$ be a finite dimensional symmetric Lie ${\infty}$-algebra.
%
Consider the \textbf{reduced graded symmetric algebra}, i.e. the graded commutative algebra $(S^{\geq 1}(E^*),\odot)$
(``reduced" means that  $S^0(E^*)= \mathbb K$ is not included in the algebra structure). Elements in $(E_{-i_1})^*\odot \ldots\odot (E_{-i_k})^*$ are said to be of \textbf{degree} $i_1+\ldots+i_k$ and \textbf{arity} $k$.
  Let $ {\mathcal E}$  be the formal completion of $E$ with respect to the arity.
  Elements of $ {\mathcal E}$ shall be referred to as \textbf{functions} and are, by definition, formal sums
   $$ F = \sum_{k \geq 1} F^{(k)} $$
   with $ F^{(k)} \in S^k(E^*)$ an element of arity $k$.
  By $F(x_1, \dots,x_k) $, with $x_1, \dots, x_k \in E$, we mean the element of ${\mathbb K}$ obtained by pairing $F_k
  \in S^{k}(E^*) \simeq (S^k(E))^*$ with $x_1 \odot \dots \odot x_k \in S^k(E)$.
  
We denote by $\mathcal{E}_i^k$, the space of functions of arity $k$ and degree $i$, i.e.
 \begin{equation} 
 \label{eq:Eik} \mathcal{E}_i^k = \oplus_{i_1+ \dots+ i_k =i } (E_{-i_1})^*\odot \ldots\odot (E_{-i_k})^* .  
 \end{equation}
 When only the degree or the arity is specified, we shall denote the corresponding vector space by $ {\mathcal E}_i$ and $ {\mathcal E}^k$, respectively. 



Since it is the completion of  $(S^{\geq 1}(E^*),\odot)$),  $\E$ has no unit.
When a unit is added (i.e., if we consider the completion of  $(S(E^*),\odot)$ instead), we obtain a unital algebra that we shall denote by $\underline{{\E}} $.

We say that a map $\Psi: \mathcal{E}\to \mathcal{E}$  has degree $i \in {\mathbb Z}$ if
$\Psi(\mathcal{E}_{\bullet})\subset \mathcal{E}_{\bullet+i}$ and arity $k \in {\mathbb Z}$ if $\Psi({\mathcal E}^\bullet) \subset {\mathcal E}^{\bullet+k}$.
Any map of degree $i$ decomposes according to arity:
 $$ \Psi = \sum_{k \in \mathbb Z} \Psi^{(k)} .$$
(Notice that the sum runs on $\mathbb Z $: a linear map can very well reduce arity.)

We call  the graded derivations of $\mathcal{E}$ {\textbf{vector fields of the graded manifold $E$}}. The  vector space $\Der (\E) $ of vector fields of $E$   is a graded Lie algebra
with respect to the bracket
\begin{equation*}
\brr{Q,P}=Q\smalcirc P - (-1)^{|Q||P|} P\smalcirc Q, \qquad Q,P\in \Der (\E) .
\end{equation*}

  A symmetric Lie $\infty$-algebra structure  $(E, \set{l_k}_{k\geq 1})$ induces a  degree $+1$ derivation of  $\mathcal{E}$,  $Q_E:\E\to \E$, squaring to zero. This derivation decomposes according to its arity:
   $$ Q_E =\sum_{k \geq 0} Q_E^{(k)} $$
where, for each $k\geq 0$, the arity $k$ derivation $Q_E^{(k)}$
 is given  by:
\begin{equation*}
Q_E^{(k)}(\xi)(x_1\odot\ldots\odot x_{k+1})=(-1)^{|\xi|}\eval{\,\xi\,,\,l_{k+1}( x_{1},\ldots, x_{k+1})\,},
\end{equation*}
 for all $ \xi\in E^*$, $x_1,\ldots, x_{k+1}\in E$.

For each homogeneous $x\in E$, let $i_x:S(E^*) \to S(E^*)$ be the derivation of arity $-1$ and degree $|x|$ defined by the  evaluation map
\begin{equation*}
i_x(\xi)=\eval{\xi,x},\quad \xi\in E^*.
\end{equation*}

The vector field $Q_E$ satisfies
\begin{equation*}
 \brr{\,\brr{\,\brr{Q_E ,i_{x_1}}, i_{x_2}},\ldots, i_{x_k}}^{(-1)}=i_{(-1)^{k+1}l_k(x_1,\ldots ,x_k)},
\end{equation*}
for all $k\geq 1$ and $x_1,x_2, \ldots, x_k\in E$.

By going backward, one derives a Lie $\infty$-algebra structure on $E$ out of a degree $+1$ derivation of the completion $ {\mathcal E}$ of $S^{\geq 1}(E^*)$, which leads to the following:

\begin{prop}\label{prop:dual}\cite{LM,V05}  Let $E$ be a graded vector space of finite dimension,
	and let  $ {\mathcal E}$ be the formal completion of $S^{\geq 1}(E^*)$ with respect to arity.
	There is a one-to-one correspondence between:
	\begin{enumerate}[(i)]
		\item  symmetric Lie $\infty$-algebra structures on the graded vector space $E$,
		\item degree $+1$ vector fields on the graded manifold $E$ (= degree $+1$ derivations of $ {\mathcal E}$) squaring to $0$.
	\end{enumerate}
 \end{prop}
\noindent
The differential graded commutative algebra  referred to in item \emph{(ii)} of Proposition \ref{prop:dual} is sometimes seen as a ``$Q$-manifold over a point" \cite{V10}.


 \begin{remark}\normalfont
 \label{rem:Qandl}
 	It is extremely practical to look at the ``abstract" correspondence explained in Proposition \ref{prop:dual} in coordinates. Given
 	$\xi^1,\xi^2,\ldots\in E^*$  a homogeneous basis of $E^*$ and $x_1,x_2,\ldots\in E$  its dual basis, the explicit formulas that relate the vector field $Q_E $ and the Lie $\infty$-algebra brackets $\set{l_k}_{k\geq 1}$ in Proposition \ref{prop:dual} can be checked to be (for our conventions):
 	$$ \left\{ \begin{array}{rcl} Q_E&=&\displaystyle \sum_j\sum_{k,i_1,\ldots, i_k} \frac{1}{k!}Q_{i_1\ldots i_k}^{j} \, \xi^{i_k} \odot\ldots\odot  \xi^{i_1}\, \frac{\partial}{\partial \xi^j},\\
  l_k&=&\displaystyle \sum_{j,i_1,\ldots, i_k} \frac{(-1)^{|x_j|}}{k!}\, Q_{i_1\ldots i_k}^{j} \, \xi^{i_k} \odot \ldots \odot \xi^{i_1}\otimes x_j, \\
 	l_k ( x_{i_1} , \dots, x_{i_k} ) &=&   \displaystyle \sum_{j=1}^n (-1)^{|x_j|} Q_{i_1\ldots i_k}^{j} x_j. \end{array}\right.$$ 
 	The coefficients	$Q^{j}_{i_1, \dots ,i_k}  \in \mathbb K$ are unique if we assume them to be graded symmetric, i.e. $ Q^{j}_{i_{\sigma(1)}, \dots ,i_{\sigma(k)}} = \epsilon(\sigma) Q^{j}_{i_1, \dots ,i_k} $ for every permutation $\sigma$.
 \end{remark}

\begin{convention}
In view of Proposition \ref{prop:dual}, we will use the notation
$\left( E, Q_E \equiv \set{l_k}_{k\geq 1}
\right)$
 to denote a Lie $\infty$-algebra of finite dimension, depending on the context.
\end{convention}

\begin{defn}
The \textbf{cohomology of a Lie $\infty$-algebra} $\left( E, Q_E \equiv \set{l_k}_{k\geq 1}
\right)$ is the cohomology defined by the differential $Q_E : \E_\bullet \to \E_{\bullet+1} $. It is a graded commutative algebra denoted by $H^\bullet(E,Q_E)$.
\end{defn}

\subsection{Morphisms of Lie ${{\infty}}$-algebras}
\label{sec:morphism}

A morphism of Lie $\infty$-algebras \cite{LM} is generally defined  as being a comorphism between symmetric coalgebras that is compatible with the Lie $\infty$-structures. When spelled out, it is equivalent to the following set of conditions. 


\begin{defn}\label{defn:Lie:infty:morphism}
Let $(E,Q_E \equiv \set{l_k}_{k\geq 1})$ and $(F,Q_F \equiv \set{m_k}_{k\geq 1})$ be Lie $\infty$-algebras. A \textbf{Lie ${\infty}$-morphism } $\ds \Phi:(E,Q_E \equiv \set{l_k}_{k\geq 1})
\rightarrow
(F,Q_F \equiv \set{m_k}_{k\geq 1})$  is given by a collection of degree zero maps:
$$
\Phi_k:S^k(E)\to F,\quad k\geq 1,
$$
such that, for each $n\geq 1$,
\begin{equation*}
\sum_{\begin{array}{c} \scriptstyle{k+l=n}\\ \scriptstyle{\sigma\in Sh(k,l)}\end{array}}\!\!\!\!\!\!\! \varepsilon(\sigma) \Phi_{1+l}(l_k\otimes 1^{\otimes^l})(x_{\sigma(I)}) = \!\!\!\!\!\!\!\!\!\!\! \sum_{\begin{array}{c}\scriptstyle{k_1+\ldots+ k_j=n} \\ \scriptstyle{\sigma\in Sh(k_1,\ldots, k_j)}\end{array}} \!\!\!\!\!\!\!  \frac{\varepsilon(\sigma)}{j!}\,m_j(\Phi_{k_1}\otimes \Phi_{k_2}\otimes\ldots\otimes \Phi_{k_j})(x_{\sigma(I)}).
\end{equation*}

\end{defn}

In Definition \ref{defn:Lie:infty:morphism}, we do not need to assume $E$ to $F$ to be of finite dimension.
When it is the case, taking the dual of the linear maps $\Phi_k$, for all $ k \geq 1$, we obtain a family $\Phi^*_k:F^* \to S^k(E^*) $ of linear maps, which extend to a graded commutative algebra morphism $\Phi^*: {\mathcal F} \to {\mathcal E}$, with $\F$ and $\E$ being the graded commutative algebras
of functions on $F$ and $E$ respectively.

This leads to the following alternative description of Lie $\infty$-morphisms:

\begin{prop}
	\label{prop:morphisms}
Let $\left( E, Q_E \equiv \set{l_k}_{k\geq 1}
\right)$ and $\left( F, Q_F \equiv \set{m_k}_{k\geq 1}
\right)$ be Lie $\infty$-algebras of finite dimensions with functions $ {\mathcal E}$
and $ {\mathcal F}$, respectively.
There is a one-to-one correspondence between:
\begin{enumerate}[(i)]
	\item 	Lie ${\infty}$-morphisms $\Phi $ from $\left( E, Q_E \equiv \set{l_k}_{k\geq 1}
	\right)$ to $\left( F, Q_F \equiv \set{m_k}_{k\geq 1}
	\right)$,
	\item graded commutative algebra morphisms $\Phi^*:\F \to \E$ commuting with vector fields:
	$$
	\Phi^*\smalcirc Q_F=Q_E \smalcirc \Phi^*.
	$$
\end{enumerate}
\end{prop}
\noindent
The second item in Proposition \ref{prop:morphisms} means that $\Phi $ is a morphism of $Q$-manifolds over a point \cite{V10}.

\begin{remark}
\normalfont
For every Lie $\infty$-morphism as above,  $\Phi_1:(E,l_1)\to (F,m_1)$ is a chain map.
The map it induces at the level of cohomology is a graded Lie algebra morphism.
Moreover, the Lie $\infty$-morphism $\Phi$ is a Lie $\infty$-isomorphism if and only if $\Phi_1$ is a chain isomorphism.
\end{remark}

\begin{defn}
Let $\ds \Phi:\left( E, Q_E \equiv \set{l_k}_{k\geq 1}
\right) \rightarrow \left( F, Q_F \equiv \set{m_k}_{k\geq 1}
\right))$ be a Lie $\infty$-algebra morphism. We say $\Phi$ is a \textbf{Lie $\infty$-quasi-isomorphism} if the chain map $\Phi_1:(E,l_1)\to (F,m_1)$ is a quasi-isomorphism.
\end{defn}


Let $\Omega^\bullet([0,1])$ stand for the de Rham complex of forms over $ [0,1]$ and let ${\rm d}_{dR}$ be its de Rham differential. Also, we shall denote by $t$ the coordinate in $[0,1]$.


\begin{defn}
	\label{def:HomotopyMorphisms}
	Let $\left( E, Q_E \equiv \set{l_k}_{k\geq 1}
	\right)$ and $\left( F, Q_F \equiv \set{m_k}_{k\geq 1}
	\right)$ be Lie $\infty$-algebras of finite dimensions with functions $ {\mathcal E}$
	and $ {\mathcal F}$. A {\textbf{homotopy}} between  Lie $\infty$-algebra morphisms $\Phi,\Psi$
	is a morphism of graded commutative algebras:
	 \begin{equation}
	 \label{eq:homotopyPhi}
\Xi^* \colon	 ({\mathcal F}, Q_F ) \longrightarrow ({\mathcal E} \otimes \Omega^\bullet([0,1]), Q_E \otimes {\rm id}  +  {\rm id} \otimes {\rm d}_{dR}) 
	 \end{equation} 
	 which coincides with $ \Phi^*$ and $\Psi^*$ at $t=0$ and $1$, respectively.
	 
	 A {\textbf{homotopy equivalence}} between $\left( E, Q_E \equiv \set{l_k}_{k\geq 1}
	\right)$ and $\left( F, Q_F \equiv \set{m_k}_{k\geq 1}
	\right)$ is a pair of Lie $\infty $-algebroid morphisms:
	\begin{align*} 	 	\Phi^* \colon ({\mathcal E}, Q_E ) \longrightarrow ({\mathcal F}, Q_F )\\ 	 \Psi^* \colon ({\mathcal F}, Q_F ) 	\longrightarrow ({\mathcal E}, Q_E ) \end{align*} 	 
	whose compositions $ \Phi^*  \circ \Psi^*  $ and $ \Psi^*  \circ \Phi^*  $ are homotopy equivalent to the identity.
\end{defn}

Let us spell out the meaning of \eqref{eq:homotopyPhi}.
To start with, recall that an element in $  V \otimes C^\infty ([0,1]) $, with $  V$ a finite dimensional vector space, can be seen as a time-dependent element in $V$ that we should denote by $ F_t \otimes 1 $.
Similarly, since $\omega \in \Omega^1 ([0,1])$ reads $f(t) dt $ with $f(t) \in C^\infty ([0,1])$, an element in $ V \otimes \Omega^1 ([0,1]) $ can be seen as a time-dependent element in $ V$ that we shall denote by $ H_t \otimes {\textrm d}t $ (the dependence in $t$ being smooth in both cases).
As a consequence, an element of degree $i$ of ${V} \otimes \Omega^\bullet([0,1])$ can be seen as a sum   $ G_t \otimes 1 + H_t \otimes {\mathrm{d}}t$, where $G_t,H_t \in V$ are elements of degree $i$ and $i-1$, respectively, which  depend smoothly on $t \in [0,1]$.

With these conventions in mind, for any algebra morphism $\Xi^*$ as in \eqref{eq:homotopyPhi} and every $F \in \mathcal F$ of degree $i$, we have 
 $$ \Xi^* (F) = G_t \otimes 1 + H_t \otimes {\mathrm d}t,   $$
 for some time-dependent $G_t\in  \mathcal E_{i}$ and $H_t \in \mathcal E_{i-1}  $. For every fixed $t \in [0,1]$, we define $\Xi^*_t,H_t: \mathcal F \to \mathcal E $ of respective degrees $0$ and $-1$ by: 
  $$ \Xi_t^* \colon F \mapsto G_t \hbox{ and } H_t \colon F \mapsto H_t. $$
  Since $\Xi^* $ is a graded algebra morphism, so is $\Xi_t^* $ for every $t \in [0,1]$, and  for all $F_1, F_2 \in \mathcal F $:
 $$ 
 H_t ( F_1 \odot  F_2) = H_t (F_1) \odot\Xi^*_t (F_2) + (-1)^{|F_1|} \Xi^*_t(F_1) \odot H_t (F_2) .
 $$
Since $\Xi^* $ is a chain map, so is $\Xi_t^* $, for all $t$ (it is therefore a family of Lie $\infty$-algebroid morphisms) and the following relation holds:
 $$ \frac{\mathrm d \Xi_t^*}{\mathrm dt}  = H_t \circ Q_F + Q_E \circ H_t .  $$




Homotopic morphisms induce the same map in cohomology, see, e.g. \cite{LLS}.
An {\bf{inverse up to homotopy}} of a  Lie $\infty$-algebra morphism $\Phi$
is a  Lie $\infty$-algebra morphism $ \Psi$ such that $\Phi \smalcirc \Psi$ and $ \Psi \smalcirc \Phi$ are homotopic to the identity maps.

\subsection{Representations of Lie $\infty$-algebras}

 A complex $(V,d)$ induces a  natural symmetric DGLA structure in $\End(V)[1]$, see Example \ref{ex:DGLA:End:E}.

\begin{defn} 
 A \textbf{representation} of a Lie $\infty$-algebra $(E,\set{ l_k}_{k\in\Zz})$ on a complex $(V,d)$ is a Lie $\infty$-morphism $$\Phi:(E,\set{ l_k}_{k\in\Zz})\rightarrow (\End(V)[1],\partial, \set{\, , \,}).$$
\end{defn}

\begin{remark}
\normalfont
\label{rem:firstversion}
 Equivalently, a representation of $E$ is defined by a collection of degree $1$ maps
$$\Phi_k:S^k(E)\to \End(V),\quad k\geq 1,$$
such that, for each $n\geq 1$, $x_1,\dots, x_n\in E$,
\begin{align}
&\sum_{\begin{array}{c} \scriptstyle{i=1}\\ \scriptstyle{\sigma\in Sh(i,n-i)}\end{array}}^{\scriptstyle{n}}\!\!\!\!\!\!\!\!\!\! \varepsilon(\sigma)\Phi_{n-i+1}\left(l_i\left(x_{\sigma(1)}, \ldots, {x_{\sigma(i)}}\right), {x_{\sigma(i+1)}}, \ldots, {x_{\sigma(n)}}\right)= \nonumber\\
&=   \partial  \Phi_n(x_1,\ldots, x_n)+ \frac{1}{2} \!\!\!\!\!\!\!\!\!\!\sum_{\begin{array}{c} \scriptstyle{j=1}\\ \scriptstyle{\sigma\in Sh(j,n-j)}\end{array}}^{\scriptstyle{n-1}} \!\!\!\!\!\!\!\!\!\!\varepsilon(\sigma)\set{\Phi_j({x_{\sigma(1)}}, \ldots, {x_{\sigma(j)}}) , \Phi_{n-j}({x_{\sigma(j+1)}}, \ldots, {x_{\sigma(n)}}) }. \label{eq:def:representation}
\end{align}
It is convenient to define $\Phi_0 $ to be the differential  $d\colon V \to V$. \end{remark}

\begin{remark}
\normalfont
\label{rem:PhikFOrRepresentation}
Equation (\ref{eq:def:representation}) may then be expressed in a more concise manner:
\begin{equation}\label{eq:def:representation:LModule}
\sum_{\begin{array}{c} \scriptstyle{i=1}\\ \scriptstyle{\sigma\in Sh(i,n-i)}\end{array}}^{\scriptstyle{n}} \!\!\!\!\!\!\!\!\!\!\varepsilon(\sigma) \, \tilde\Phi_{n-i}(\tilde\Phi_{i-1}(y_{\sigma(1)}, \ldots, y_{\sigma(i)}), y_{\sigma(i+1)}, \ldots, y_{\sigma(n)})=0,
\end{equation}
where $y_i = x_i + v_i \in  E \oplus V$ for all index, and where
$$  \tilde\Phi_{i}( x_1 + v_1,  \dots x_i + v_i ) = l_i(x_1, \dots, x_i) + \sum_{k=1}^i \epsilon_k   \Phi_{i-1}(x_1 , \dots,\widehat{x_k} , \dots x_i  ) (v_k),$$
with $\epsilon_k= (-1)^{|v_k|(|x_{k+1}|+ \dots + |x_{i}| )}$.
%
%
%
%
Equations \eqref{eq:def:representation:LModule} mean that the family $\tilde{\Phi} $  equips  $E \oplus V$ with a Lie $\infty$-algebra structure, called the \textbf{semi-direct product} of a Lie $\infty $-algebra with its representation. For  $\mathfrak g $ a Lie algebra and $V$ a vector space, we recover the usual semi-direct product Lie algebra $ \mathfrak g \oplus V $.
\end{remark}

Let us give the dual description.
The tensor product $\E \otimes V^*$ comes with a natural left $\E$-module structure given by $F \cdot (G \otimes \al):= (F \odot G) \otimes \al$ for all $F,G \in \E$, $\al\in V^*$.

\begin{prop} \label{repres:1to1}
	 Let $ \left(E, Q_E \equiv \set{l_k}_{k\geq 1} \right)$ be a Lie $\infty$-algebra of finite dimension, and $ \underline \E$ be its unital graded commutative algebra of functions.
	 Let $V$ be a graded vector space.
	 There is a $1$-$1$ correspondence between:
	 \begin{enumerate}[(i)]
	 	\item representations of
	 	$ \left(E, Q_E \equiv \set{l_k}_{k\geq 1} \right)$  on the complex $(V,d)$,
	 	\item degree $+1$ maps  $\mathcal{D}:{\mathcal E}\otimes V^*\to{\mathcal E}\otimes V^*$,
	 	 \begin{enumerate}
	 	 	\item which extends $d^*: V^* \to V^*$, in the sense that the component of $ \mathcal{D}(1 \otimes \beta)$ in $1 \otimes V^*$ is $ 1 \otimes d^*(\beta)$, for all $\beta \in V^*$ ,
	 	 	\item which squares to $0$, i.e. $\mathcal{D}^2=0$,
	 	 	\item which is compatible with $Q_E$ in the sense that:
	 	 		$$
	 	 		\mathcal{D}(F\cdot \psi )=Q_E(F)\cdot \psi + (-1)^{i}F\cdot \mathcal{D}\psi, \quad F \in {\E}_i, \psi \in {\mathcal E}\otimes V^*.
	 	 		$$
	 	 	\end{enumerate}
	 \end{enumerate}
\end{prop}
\begin{proof}
    The dual of the semi-direct product Lie $\infty $-algebra structure on $E \oplus V $, defined in Remark \ref{rem:PhikFOrRepresentation}, gives a derivation $ Q_{E \oplus V} $ of $S(E^* \oplus V^* ) = \underline{\mathcal E} \otimes S(V^*) $ which restricts to ${\mathcal E} \oplus \underline{\mathcal E} \otimes V^* $ (because all brackets on the semi-direct product  are zero when two elements in $V$ are considered). This restriction $\mathcal D $ satisfies the required conditions. This construction can be reverted. 
\end{proof}

\noindent
The object in item \emph{(ii)} of Proposition \ref{repres:1to1} deserves to referred to as a flat connection.

For $V$ a representation of $\left( E, Q_E \equiv \set{l_k}_{k\geq 1} \right)$, its dual $V^*$ comes with a natural representation structure: it suffices to compose the maps $\Phi_k : S^k(E) \to {\mathrm{End}}(V)[1] $, $k\geq 0$, as in Remark
\ref{rem:firstversion}, with the natural symmetric graded Lie algebra morphism
 $$ {\mathrm{End}}(V)[1] \to {\mathrm{End}}(V^*)[1] $$
 that sends a linear map $\phi $ to
  $-\phi^* $. We call this representation the \textbf{dual representation}.

\paragraph{\textbf{Adjoint and coadjoint representation}}

An important example of a representation is given by a Lie $\infty$-algebra adjoint representation on itself. In fact, we first define its dual, i.e. the coadjoint representation.

As before, let $\left( E, Q_E \equiv \set{l_k}_{k\geq 1} \right)$ be a Lie $\infty$-algebra of finite dimension with functions $ {\mathcal E}$. Let us add a unit to $ {\mathcal E}$, hence obtaining a grade commutative algebra $\underline{{\E}} $.  The space $\Der (\E) $  of derivations of $\E$, i.e. vector fields of the graded manifold $E$,   is naturally identified with $\underline{{\E}} \otimes E$.

The  Lie bracket $\brr{{Q_E}, \, -\,}:\Der (\E)\to \Der (\E)$ satisfies all four conditions in item (ii) in Proposition \ref{repres:1to1}. Therefore, it defines a representation of the Lie $\infty$-algebra $\left( E, Q_E \equiv \set{l_k}_{k\geq 1} \right)$  on $(E^*,-l_1^*)$ 
that we call \textbf{coadjoint representation}. 

Let us spell out its dual representation.
It is routine to check that 
the following collection of degree $+1$ maps:
\begin{equation} \label{eq:adjoint:representation:algebra}
\begin{array}{rrcl}
\ad^{(k)}:& S^k(E) &\to&  \End(E) \\
 & \;x_1\odot\ldots \odot x_k & \mapsto &  \ad^{(k)}_{x_1 \odot \dots \odot x_k} := l_{k+1}\left( x_1,\ldots, x_k, \, \, \, -\,\,  \right), 
\end{array}\quad k\geq 1,
\end{equation}
defines $\ad=\sum_{k\geq 1}\ad^{(k)}$,  the representation of the Lie $\infty$-algebra $E$ on $(E,l_1)$ dual to the coadjoint representation. 

\begin{defn}
The representation $\ad$  is called   the \textbf{ adjoint representation} of  $\left( E, Q_E \equiv \set{l_k}_{k\geq 1} \right)$. 
\end{defn}


\section{The modular class of a Lie $\infty$-algebra}
\label{LieInftyModular2}

Let  $\left( E, Q_E \equiv \set{l_k}_{k\geq 1} \right)$ be a Lie $\infty$-algebra of finite dimension.

 The restriction of the adjoint representation $\ds \ad:S(E)\to \End(E)[1]$ given by (\ref{eq:adjoint:representation:algebra}) to the space $S(E)_{-1}$ of degree $-1$ elements of $S(E)$,  is a  linear map
 $$
 {\ad}_{|S(E)_{-1}}:S(E)_{-1} \subset S(E)\to \End_0(E).
 $$
 In other words, for all homogeneous $x_{1} ,\dots, x_{k}  \in E$ whose degree add up to $ -1$, $\ad^{(k)}_{x_1 \odot \dots \odot x_k}$ is a linear endomorphism of $E$ of degree $0$,
 i.e. a family indexed by $ i \in {\mathbb Z}$ of linear endomophisms of $E_{i}$.
For each $k\geq 1$, let $\omega^{(k)}\in \mathcal{E}_{1}^k$ be defined, for all homogeneous $x_{1} ,\dots, x_{k}  \in E$ whose degree add up to $ -1$, by
 $$ \omega^{(k)}( x_{1},\ldots, x_{k})=\str \ad^{(k)}_{x_{1} \odot \ldots \odot x_{k}}, $$
where $\str$ denotes the super trace operator on $E$, i.e. the alternate sum of the traces
of the restriction to $E_i$ of $\ad^{(k)}_{x_{1} \odot \ldots \odot x_{k}}$.

\begin{defn}\label{def:modular} We call \textbf{modular function} of a  Lie $\infty$-algebra $\left( E, Q_E \equiv \set{l_k}_{k\geq 1} \right)$ of finite dimension, the function  $\omega_E \in \E_{+1}=S(E^*)_{+1}$ defined by
 $$\omega_E=\sum_{k \geq 1} \omega^{(k)},$$
 with $\omega^{(k)} \in S^{k}(E^*)_{+1}$.
\end{defn}

\begin{remark}
\normalfont
It deserves to be noticed that we not really need $E$ to be of finite dimension. To define the modular function, it is enough to assume it is of finite rank in every degree. It even suffices that the super-trace is defined.  
\end{remark}

Let $\xi^1, \dots,\xi^n$ be a basis of $ E^*$, made of homogeneous elements. 	Using conventions of Remark \ref{rem:Qandl}, any derivation $Q$ of $ {\E}$, of degree $i$ reads:
 $$ Q=  \sum_{k \in {\mathbb N}}\!\!\!\!\! \sum_{  {\tiny{\begin{array}{c} 1 \leq i_1,\dots,i_k \leq n \\
 	j=1, \dots, n \\ |\xi^{i_1}| + \dots + |\xi^{i_k}| - |\xi^{j}|= i \end{array}}}} \!\!\!\!\!\!\!\!\!\! \frac{1}{k!}\,\,\, Q^{j}_{i_1, \dots ,i_k} \xi^{i_k} \odot \dots \odot \xi^{i_1} \, \, \, \frac{\partial}{\partial \xi^j}   .$$
We define its \textbf{divergence} to be the element in $ \underline{\E} $ given by:
\begin{align*} {\rm div} (Q) &:=(-1)^{i+1}\sum_{j=1}^n \frac{\partial}{\partial \xi_j}Q(\xi_j)\\
 &=(-1)^{i+1} \sum_{k \in {\mathbb N}} \!\!\!\!\!\!\! \sum_{  {\tiny{\begin{array}{c} 1 \leq i_1,\dots,i_k \leq n \\
			j=1, \dots, n \\ |\xi^{i_1}| + \dots + |\xi^{i_k}| - |\xi^{j}|= i \end{array}}}} \!\!\!\!\!\!\! \frac{1}{k!} Q^{j}_{i_1, \dots ,i_k} \frac{\partial}{\partial \xi^j} \left( \xi^{i_k} \odot \dots \odot \xi^{i_1} \right).\\
&=  (-1)^{i+1}\sum_{k \in {\mathbb N}} \!\!\!\!\!\!\! \sum_{  {\tiny{\begin{array}{c} 1 \leq i_1,\dots,i_{k-1} \leq n \\
			j=1, \dots, n \\ |\xi^{i_1}| + \dots + |\xi^{i_{k-1}}| = i \end{array}}}} \!\!\!\!\!\!\! \frac{1}{(k-1)!} Q^j_{i_1, \dots ,i_{k-1},j} \xi^{i_{k-1}} \odot \dots \odot \xi^{i_1}  \end{align*}
			
The next proposition implies that our definition of the modular class coincides with the ones of \cite{AB} and \cite{V07}.			
			
\begin{prop}
	\label{prop:divergence}
Let $\left( E, Q_E \equiv \set{l_k}_{k\geq 1} \right)$  be a Lie $\infty$-algebra of finite dimension.
\begin{enumerate}
 \item
The modular function $\omega_E$ is  $Q_E$-closed, i.e. $Q_E(\omega_E)=0$.
 \item
The modular function coincides with the divergence of the vector field $Q_E$:
$$\omega_E=\mathrm{div }( Q_E).$$
\end{enumerate}
\end{prop}

\begin{proof}
Let us prove the first item. By construction, $Q_E(\omega_E) $ is a function of degree $2$.
Let us compute its component of arity $k$.
Let $x_1,\ldots, x_k\in  E$ be homogeneous elements whose degrees add up to ${-2}$, then:
\begin{align*}
& Q_E(\omega_E)(x_1\odot\ldots\odot x_k) \\ =&\!\!\!\!\!\!\!\!\!\!\sum_{\begin{array}{c}\scriptstyle{h\leq k}\\ \scriptstyle{\sigma\in Sh(h,k-h)}\end{array}} \!\!\!\!\!\!\!\!\!\!(-1)^{|\omega_E|}\varepsilon(\sigma)\,\omega^{(k)}\left(l_{h}(x_{\sigma(1)},  \ldots,  x_{\sigma(h)}) \odot x_{\sigma(h+1)} \odot \ldots \odot x_{\sigma(k)} \right)\\
=&-\!\!\!\!\!\!\!\!\!\!\sum_{\begin{array}{c}\scriptstyle{h\leq k}\\ \scriptstyle{\sigma\in Sh(h,k-h)}\end{array}}  \!\!\!\!\!\!\!\!\!\! \varepsilon(\sigma) \str\, \ad^{(k)}_{l_{h}(x_{\sigma(1)},  \ldots,  x_{\sigma(h)})  \odot x_{\sigma(h+1)}\odot \ldots \odot x_{\sigma(k)} }\\
=&
  \str \left( l_1\smalcirc \ad_{x_1\odot \ldots \odot x_k}+ \ad_{x_1\odot \ldots \odot x_k}\smalcirc l_1\right)  + \\
  & + \frac{1}{2}\!\!\!\!\!\!\!\!\!\!\sum_{\begin{array}{c} \scriptstyle{h\leq l}\\ \scriptstyle{\sigma\in Sh(h,l-h)}\end{array}}\!\!\!\!\!\!\!\!\!\! (-1)^{|x_{\sigma(1)}|+\ldots + |x_{\sigma(h)}|}\str \brr{\ad^{(h)}_{x_{\sigma(1)}\odot  \ldots \odot x_{\sigma(h)}},\ad^{(k-h)}_{x_{\sigma(h+1)}\odot  \ldots  \odot x_{\sigma(k)}}}  =0.
\end{align*}
Above, we used  in the second line, Definition \ref{def:modular}  of the modular function and in the third line the description given in Equation (\ref{eq:def:representation}) of representations. In the last line, we used the fact that the supertrace of a graded commutator vanishes.

Let us prove the second item.
Let $x_1,x_2,\ldots, x_n$ and $\xi^1,\xi^2,\ldots,\xi^n$ be dual basis of $E $ and $E^*$, made of homogeneous elements.
Now, for each $k\geq 1$ and $x_{i_1}\odot\ldots \odot x_{i_k}\in S^k{(E)}_{-1}$ we have, by definition of the supertrace,
\begin{align*}\omega_E(x_{i_1}\odot\ldots \odot x_{i_k})&=\str \left( \ad^{(k)}_{x_{i_1}\odot\ldots\odot x_{i_k}} \right) \\
&= \sum_{j=1}^n (-1)^{|\xi^j|} \eval{\xi^j, l_{k+1}(x_{i_1},\ldots, x_{i_k}, x_{j})}\\
&= \sum_{j=1}^n \frac{\partial}{\partial \xi^j} Q(\xi^j)(x_{i_1}\odot\ldots \odot x_{i_k})\\
&=\left( \mathrm{div } \, Q_E \right) (x_{i_1}\odot\ldots \odot x_{i_k}).
\end{align*}
Remark	\ref{rem:Qandl} was used to relate the second to the third line.
This completes the proof.
\end{proof}

\begin{remark}
\normalfont
In the supermanifold case, Bruce \cite{AB}
proves an equivalent statement of Proposition \ref{prop:divergence}  by showing first that $ \mathcal L_{Q_E} \rho = \left( \mathrm{div } \, Q_E \right) \rho $, for any Berezinian form $ \rho$. He then notices that the relation $\mathcal L_{Q_E}^2=\frac{1}{2}\mathcal L_{[Q_E,Q_E]} =0 $ gives immediately $Q_E[\mathrm{div } \, Q_E]=0 $. This is an efficient method, but we prefer to go through the properties of the supertrace in order to have a proof that may remain valid if there exists infinitely many brackets, as in Example \ref{ex:Atiyah}. This seems complicated using \cite{AB} (which uses a formalism that does not immediately allow infinite sums). Also, the idea of relating this construction to the adjoint representation seems to be conceptually relevant. Last, we want to avoid Berezinians - at this point.
\end{remark}

The first item in Proposition \ref{prop:divergence} makes sense of the following

\begin{defn}
The \textbf{modular class of a Lie $\infty$-algebra} $\left( E, Q_E \equiv \set{l_k}_{k\geq 1} \right)$ of finite dimension is the cohomology class of the modular function $\omega_E$ in  $H^1(E,Q_E)$, the first degree cohomology of $E$.
\end{defn}

Let $\left( E, Q_E \equiv \set{l_k}_{k\geq 1} \right)$ and  $\left( F, Q_F \equiv \set{m_k}_{k\geq 1} \right)$ be Lie $\infty$-algebras of finite dimensions, and  $ \Phi$ a morphism of Lie $\infty$-algebras from the first one to the second one.
Let $ \Phi^* : \F \to \E$ be the corresponding morphism of Lie $\infty$-algebras as in Proposition \ref{prop:morphisms}. Since $\Phi^*$ is a chain map, $\omega_\Phi : = \omega_E - \Phi^* (\omega_F) $ is a cocycle of degree $1$.  As in \cite{KLW08}, we call its corresponding class in $H^\bullet(E,Q_E)$ the \textbf{modular class of the Lie $\infty$-algebra morphism~$\Phi$}.

\begin{prop}\label{prop:modular:class:isomprphism}
 The modular class of a Lie $\infty$-isomorphism vanishes.
\end{prop}

\begin{proof} This proof is surprisingly complicated, and will make use of co-derivations, and Sweedler's notation, see \cite{Vallette}.
Consider  the reduced symmetric algebra $S^{\geq 1}(E)$ equipped with its usual coalgebra structure given by:
$$
\Delta(x_1\odot\ldots\odot x_n)= \sum_{\begin{array}{c} \scriptstyle{j=1}\\ \scriptstyle{\sigma\in Sh(j,n-j)}\end{array}}^{\scriptstyle{n-1}} \!\!\!\!\!\!\!\!\!\!\varepsilon(\sigma)({x_{\sigma(1)}}\odot \ldots\odot {x_{\sigma(j)}}) \otimes ({x_{\sigma(j+1)}}\odot \ldots \odot{x_{\sigma(n)}}),
$$
for $x_1,\ldots, x_n\in E$.
We will use Sweedler's notation: 
given $x \in  S^{\geq 1}(E)$,  $$\Delta^{(1)}(x)=\Delta(x)=x_{(1)}\otimes x_{(2)},$$
 and {the coassociativity yields} $$\Delta^{(n)}(x)=(\mathrm{id}\otimes\Delta^{(n-1)})\Delta(x)=x_{(1)}\otimes \ldots\otimes x_{(n+1)},\quad n\geq 2.$$ 
Notice that $\Delta^{(n)}(x)=0$, for all $x\in S^{\leq n}(E)$.

Let $\Phi:\left( E, Q_E \equiv \set{l_k}_{k\geq 1} \right)\to \left( F, Q_F \equiv \set{m_k}_{k\geq 1} \right)$ be  Lie $\infty$-isomorphism and  
  $M_E:S^{\geq 1}(E)\to S^{\geq 1}(E)$  the coderivation of $S^{\geq 1}(E)$ defined by the Lie $\infty$-structure in $E$:
$$
M_E(x)=l(x_{(1)})\odot x_{(2)}+l(x), \quad x \in  S^{\geq 1}(E).
$$

For each $x\in S^k(E)$, $k\geq 1$, let us consider the linear map
$\Phi_x:E\to F$ given by $$\Phi_x(y)=\Phi_{k+1}(x\odot y).$$

The map
$\Phi_1:(E,l_1)\to (F,m_1)$ is an isomorphism:
$$
 l_1=\Phi_1^{-1}\smalcirc m_1\smalcirc \Phi_1
$$
and, taking into account Definition \ref{defn:Lie:infty:morphism}, we have, for each homogeneous $x\in S^{\geq 1}(E)$:
\begin{align}\label{eq:Lie:isomorphism:adjoint}
\nonumber\ad^E_x&+\Phi_1^{-1}\smalcirc\Phi_{M_E(x)} + (-1)^{|x|}\Phi_1^{-1}\smalcirc\Phi_x\smalcirc l_1 + (-1)^{|x_{(1)}|}\Phi_1^{-1}\smalcirc \Phi_{x_{(1)}}\smalcirc \ad^E_{x_{(2)}}\\
& =\Phi_1^{-1}\ad^F_{\Phi(x)}\smalcirc \Phi_1 + \Phi^{-1}_1\smalcirc m_1\smalcirc \Phi_x + \Phi_1^{-1} \smalcirc \ad^F_{\Phi(x_{(1)})}\smalcirc \Phi_{x_{(2)}}.
\end{align}

Then, for each $x\in  S^{\geq 1}(E)_{-1}$, we have
\begin{eqnarray*}
\Phi^*\omega_F(x)&=&\omega_F(\Phi(x))=\str \ad^F_{\Phi(x)}=\str \Phi^{-1}_1\smalcirc \ad^F_{\Phi(x)} \smalcirc \Phi_1\\
&=&\str \ad^E_x+\str \Phi_1^{-1}\smalcirc\Phi_{M_E(x)}-\str \left(\Phi_1^{-1}\smalcirc\Phi_x\smalcirc l_1 + \Phi^{-1}_1\smalcirc m_1\smalcirc \Phi_x\right) \\
& &+ \str (-1)^{|x_{(1)}|}\Phi_1^{-1}\smalcirc \Phi_{x_{(1)}}\smalcirc \ad^E_{x_{(2)}}
-\str \Phi_1^{-1} \smalcirc \ad^F_{\Phi(x_{(1)})}\smalcirc \Phi_{x_{(2)}}\\
&=&\omega_E - Q_E(\str \Phi_1^{-1}\smalcirc \Phi_{x})  \\
& &+\str (-1)^{|x_{(1)}|}\Phi_1^{-1}\smalcirc \Phi_{x_{(1)}}\smalcirc \ad^E_{x_{(2)}}
-\str \Phi_1^{-1} \smalcirc \ad^F_{\Phi(x_{(1)})}\smalcirc \Phi_{x_{(2)}}.
\end{eqnarray*}

By recursively applying equation (\ref{eq:Lie:isomorphism:adjoint}) we get
\begin{align*}
\Phi^*\omega_F(x) &= \omega_E + Q_E \left(\str \sum_{n}\frac{(-1)^n}{n} (\Phi_1^{-1}\smalcirc \Phi)^n_{x} \right), 
\end{align*}

where, for each $x\in S^k(E)$, $k\geq 1$,
\begin{align*}\sum_{n=1}^k \frac{(-1)^n}{n}(\Phi_1^{-1}\smalcirc \Phi)^n_{x}=& 
-\Phi_1^{-1}\smalcirc \Phi_x + \frac12 \Phi_1^{-1}\smalcirc \Phi_{x_{(1)}}\smalcirc \Phi_1^{-1}\smalcirc \Phi_{x_{(2)}}
\\
&
+ \ldots + \frac{(-1)^k}{k}
 \Phi_1^{-1}\smalcirc \Phi_{x_{(1)}}\smalcirc \ldots \smalcirc \Phi_1^{-1}\smalcirc \Phi_{x_{(k)}}
 .
\end{align*}
Therefore $\Phi^*\left[\omega_F\right]=\brr{\omega_E}.$
\end{proof}

A Lie $\infty$-algebra $\left( E, Q_E \equiv \set{l_k}_{k\geq 1} \right)$ is called \textbf{minimal} if $l_1=0$ and  is called \textbf{linear contractible} if $l_k=0$, $k\geq 2$, and $H^\bullet(E,l_1)$ is trivial.

Any Lie $\infty$-algebra is isomorphic to the direct sum of a minimal Lie $\infty$-algebra with a linear contractible one \cite{Kontsevich}. In this decomposition, the minimal Lie $\infty$-algebra is unique up to isomorphism, and a homotopy equivalence between Lie $\infty $-algebras induce an isomorphism between their minimal Lie $\infty$-algebras. 
 Since a linear contractible Lie $\infty $-algebra has a trivial modular class, the following statement follows from 
Proposition \ref{prop:modular:class:isomprphism}.


\begin{prop}
	\label{prop:up_to_homotopy}
A homotopy equivalence  of Lie $\infty$-algebras intertwines their modular classes. In equation:
 $$ \Phi^* [\omega_E] = [\omega_F]  \hbox{ and } \Psi^* [\omega_F]  =  [\omega_E],$$
 for any homotopy equivalence as in Definition \ref{def:HomotopyMorphisms}.
\end{prop}

Let $\left( E, Q_E \equiv \set{l_k}_{k\geq 1} \right)$ be a Lie $\infty$-algebra of finite dimension. Recall from Proposition {\ref{prop:cohom}} that the cohomology $H^\bullet (E,l_1)$ of the complex $(E,l_1)$ comes equipped with a natural graded Lie algebra structure. It admits therefore also a modular function, which, by construction, is of arity one, which, for degree reasons,  is an element in $(H^{-1} (E,l_1))^*$.

\begin{prop}
\label{prop:NegativeconcentrateModular}
Let $\left( E, Q_E \equiv \set{l_k}_{k\geq 1} \right)$ be a
Lie $\infty$-algebra of finite dimension.
The modular function $\omega_E$ of $E$ and the modular function  $\omega_{H(E,l_1)}$ of the graded Lie algebra $H^\bullet(E,l_1)$ are related by:
\begin{equation*}
\omega_{H^\bullet(E,l_1)}\left(\brr{x}\right)=\omega_E(x),
\end{equation*}
for all  $ x \in E_{-1} $ with $l_1(x)=0$.
\end{prop}

\begin{proof} 
	For every graded Lie algebra, the modular function is the supertrace of the adjoint action on elements of degree $-1$. 
	
	In particular,
	the modular function of  the graded Lie algebra $H^\bullet(E,l_1)$ is
	$$ \omega_{H^\bullet(E,l_1)}\left(\brr{x}\right)=\str \left( \ad_{\brr{x}}^{(2)}\right),\quad \brr{x}\in H^{-1}(E,l_1).$$
	Now, for a complex $(E,l_1)$ of finite dimension, the supertrace of a chain map is equal to the supertrace of the map it induces in cohomology on $H^\bullet (E,l_1)$.
	Hence, for all  $ x \in E_{-1} $ with $l_1(x)=0$.
	$$ \str \left( \ad_{\brr{x}}^{(2)} \right)=\str \left( \ad_x^{(2)} \right).$$
This completes the proof.
\end{proof}

\begin{remark}\normalfont
	Proposition \ref{prop:up_to_homotopy} allows to make sense of  Lie $\infty$-algebras
	of infinite dimension which are homotopy equivalent to one of finite dimension. By the homotopy transfer theorem, these are exactly Lie $\infty$-algebras $(E, \set{l_k}_{k\geq 1} )$
	such that the cohomology of the complex $(E,l_1)$ has finite dimension.
\end{remark}



\begin{ex}
For any representation  $\Phi: S(E)\to  \End(V)[1]$ of a Lie $\infty$-algebra $(E,Q_E\equiv\set{l_k}_{k\geq 1})$ on the complex $(V,d)$, 
consider the degree $+1 $ function $\omega_\Phi \in \mathcal E_{+1} $ given by  $$ \omega_\Phi(x_1, \dots, x_k )= \str \Phi_k (x_1,\ldots, x_k) ,$$ for each homogeneous
$ x_1, \dots, x_k \in E$, $k\geq 1$, whose degrees add up to $-1$. This function is $ Q_E$-closed, hence define a class in $H^\bullet(E,Q_E)$, that we call the  \textbf{characteristic class of the representation $\Phi$}.
Now, the direct sum $E\oplus V$ comes equipped with the semi-direct product Lie $\infty$-algebra structure (see Remark \ref{rem:PhikFOrRepresentation}). We leave it to the reader to check that the characteristic class of the representation $\Phi$ is the modular class of the inclusion $ E \hookrightarrow E \oplus V$.
For the adjoint representation, we recover the modular function of Definition  \ref{def:modular}.
%
%
%
%
%
\end{ex}

\begin{ex}
\label{ex:negGraded}
For Lie $\infty $-algebras concentrated in degree less or equal to $ -1$ (i.e. for "negatively graded Lie $\infty$-algebras"), only the $2$-ary bracket with an element of degree $-1$ can contribute to the modular function. In view of Proposition \ref{prop:NegativeconcentrateModular}, we see that its modular function is entirely determined by the one of its associated symmetric GLA $H^\bullet(E,l_1) $, although, in general, a Lie $\infty $-algebra is not homotopy equivalent to its associated GLA $H^\bullet(E,l_1) $. 
\end{ex}

\begin{ex}
Consider the following Lie $\infty $-algebra: $E_{-2}=\mathbb R, E_0= \mathbb R,E_{1}= \mathbb R $ and $ E_i=0$ otherwise. All brackets are $0$ except the $ 3$-ary brackets of the generators $e_{-2}\in E_{-2}$, $e_0\in E_0$ and $e_1\in E_1$, for which 
 $$  \{ e_{-2},e_1, e_0 \}_3= e_0 .$$
The modular function is defined on generators by
 $$ \omega ( e_{-2}, e_1) =1  $$ 
 and is zero otherwise. It is, in particular, non-zero. 
 This example shows that the $3$-ary bracket may contribute to the modular function, in contrast with the situation of Example \ref{ex:negGraded}.
\end{ex}

\begin{ex}[Lie algebra actions]
\label{ex:Atiyah}\normalfont
	Let $G$ be a Lie group with Lie algebra $ {\mathfrak g}$ acting on a vector space $V$.
	We do not assume that the action is linear, but we assume the origin $O$ to be a fixed point of this action, as in \cite{LSX}, Section 4.3. The Lie algebra ${\mathfrak g}$ then acts
	by derivation on the infinite-jet $\widehat{S}(V)$ of functions at $0$. The Chevalley-Eilenberg differential of this action is a derivation $Q$ of degree $+1$ squaring to zero of 
	  $ \wedge^\bullet {\mathfrak g}^* \otimes \widehat{S}(V) $.
	  Explicitly:
	   $$ Q = \sum_{i,j,k= 1}^n \Gamma_{ij}^k \, \epsilon^i \wedge \epsilon^j \, \frac{\partial}{\partial \epsilon_k}   + \sum_{i=1}^n \epsilon^i \otimes X_i^{(\infty)} $$
	   where $ (\epsilon_i)_{i=1, \dots,n}$ is a basis of $ {\mathfrak g}^*$, $ (e_i)_{i=1, \dots,n} $ is the dual basis and $X_i^{(\infty)}$
	    is the infinite jet at $m$ of the formal vector field (that is to say, an element in $\widehat{S}(V^*) \otimes V$) associated to the infinitesimal action of $e_i$ on $V$.
	
	  By Proposition \ref{repres:1to1}, it therefore corresponds to a Lie $\infty$-algebra structure concentrated in degrees $-1$ and $0$ on ${\mathfrak g}|_{-1} \oplus V|_0$.
	  According to Theorem \ref{prop:divergence}, the modular function of this Lie $\infty$-algebra is:
	   \begin{eqnarray*} \omega_E &=& \sum_{i=1}^n \sum_{j=1}^n \Gamma_{ij}^j \epsilon^i  +  \sum_{i=1}^n {\rm div} (X_i^{(k)})   \, \epsilon_i \\ &=& \omega_{\mathfrak g} + \sum_{i=1}^n {\rm div} (X_i^{(\infty)}) \, \epsilon_i \end{eqnarray*}
	   where $\omega_{\mathfrak g} \in {\mathfrak g}^*$ is the modular function of the Lie algebra $ {\mathfrak g}$, and where  ${\rm div} (X_i^{(\infty)})$ is the divergence of the infinite jet of the vector field $X_i$.
	   This class is zero if $\mathfrak g $ is a unimodular Lie algebra, and if there exists an infinite jet of a volume form on $M$ formally preserved by the $\mathfrak g$-action. 
\end{ex}

\section{Lie ${\infty}$-algebroid}\label{LieInfty}

\subsection{Lie $\infty$-algebroids and their cohomologies}

In what follows,  $M$ is a manifold whose sheaf of functions we denote by $\mathcal{O}$.

\begin{remark}
\normalfont
Unless otherwise specified, $M$ can be a complex or a smooth manifold. It can also be an affine variety, with $\mathcal O$ then being the sheaf of  its regular functions.
\end{remark}

We warn the reader not to confuse $\mathbb Z_2 $-grading used by Bruce \cite{AB} or in several works by Khudaverdian and Voronov \cite{KV1,KV2,KV3,V07} and the $\mathbb Z $-grading we are using here.

\begin{defn} \cite{LM,V05} A \textbf{negatively graded Lie $\infty$-algebroid} is a collection of vector bundles $E=\oplus_{i\geq 1} E_{-i}$ over $M$ equipped with a sheaf of Lie $\infty$-algebra structures over the sheaf of sections and a vector bundle morphism
$\rho : E_{-1} \to TM$, called the {\textbf{anchor}} of $E$, such that the brackets $l_k=\set{-\, , \ldots,\,-}_k$ are all $\mathcal{O}$-linear in each argument except if $k=2$ and at least one of its two entries has degree $-1$. In the latter case, however:
\begin{equation*} \set{x, fy}_2 = f\set{x, y}_2 + \rho(x)[f] y,\quad   x\in \Gamma( E_{-1} ), y  \in\Gamma(E).
\end{equation*}
\end{defn}

It follows from the definition that $\ds \rho\smalcirc l_1(x)=0$   and that $\rho ( \{y,z\}_2) = [\rho(y),\rho(z)]$ for all $ x \in \Gamma(E_{-2})$ and $ y,z \in \Gamma(E_{-1})$.

\begin{ex}
	When $M$ is a point, we recover negatively graded Lie $\infty$-algebras studied in the previous section. When $E_{-i}=0$ for $i \neq 1$, we recover Lie algebroids.
 \end{ex}


\begin{prop}\label{prop:dual2}\cite{LM,V05}  Let $E=\oplus_{i \geq 1}E_{-i}$ be a graded vector bundle over a manifold $M$
	and let  ${\mathcal E}$ be the sheaf of sections of the graded symmetric algebra bundle $ S(E^*)$ (with the understanding that $E_{-i}^*$ is considered to be of degree $+i$).
	There is a one-to-one correspondence between:
	\begin{enumerate}[(i)]
		\item negatively graded Lie $\infty$-algebroid brackets and anchor $((l_k)_{k \geq 1}, \rho)$ on the graded vector bundle $E$,
		\item degree $+1$ vector fields $Q$ on the graded manifold $E$ (= degree $+1$ derivations of $ {\mathcal E}$) squaring to $0$.
	\end{enumerate}
\end{prop}

In the proposition above, the sheaf $ {\mathcal E}$ will be referred to as \emph{functions on $E$}. Again,
sections  of $E^*_{-i_1}\otimes\ldots\otimes E^*_{-i_k}$ are said to be functions of \textbf{degree} $i_1+\ldots+i_k$ and \textbf{arity} $k$.
We denote by $\mathcal{E}_i$, the space of functions of degree $i$ and by $\E^{k}$ the space of functions of arity $k$, as in Section \ref{sec:arity}.
We denote by ${\mathfrak X} (E) $ the graded Lie algebra of derivations of $\mathcal{E} $, that shall be also referred to as {\textbf{vector fields}} on $E$.

We say a map $\Psi: \mathcal{E}\to \mathcal{E}$ has  arity $k$ if $\Psi(\E^\bullet)\subset \E^{\bullet+k}$. Also we say it has degree $p$ if
	$\Psi(\mathcal{E}_{\bullet})\subset \mathcal{E}_{\bullet+p}$.

The pair $(\mathcal E, Q) $ in item \emph{(ii)} of Proposition \ref{prop:dual2} is often referred to as a {\textbf{$Q$-manifold with a trivialization}}. Proposition \ref{prop:dual2} allows the following convention:

\begin{convention}
From now on, we shall say simply 
``Lie $\infty$-algebroids" for  ``negatively graded Lie $\infty$-algebroids".
	Also, Lie $\infty$-algebroids shall, from now, be denoted as 
	$$(E,Q \equiv (\set{l_k}_{k \geq 1}, \rho)),
	$$ where $E = (E_{-i})_{i \geq 1}$ is a graded vector bundle over $M$, $\set{l_k}_{k \geq 1} $ are the Lie $\infty$-brackets, $\rho $ is the anchor map and $Q : {\mathcal E} \to {\mathcal E}$ is
	the degree $+1$ derivation - related as in items (i) and (ii) in Proposition \ref{prop:dual2}.
\end{convention}


Here are two important notions:

\begin{defn}
	Let $(E,Q\equiv(\set{l_k}_{k\geq 1},\rho))$ be a Lie $\infty$--algebroid over $M$. The \textbf{linear part}  $(E\oplus TM, l_1,- \rho)$  is the complex of vector bundles over $M$ defined by
	$$
	\ldots \xrightarrow{l_1} E_{-n} \xrightarrow{l_1} \ldots \rightarrow E_{-2} \xrightarrow{l_1} E_{-1} \xrightarrow{-\rho} TM.
	$$
\end{defn}

In addition to the cohomology of the linear part,  the following cohomology is also important \cite{Kotov,Salnikov}:

 \begin{defn}
 	The \textbf{cohomology of a Lie $\infty$-algebroid} $(E,Q\equiv(\set{l_k}_{k\geq 1},\rho))$ is the cohomology $H^\bullet(E,Q)$ of the graded commutative algebra $ {\mathcal E}$ with respect to the differential $Q$.
 \end{defn}

 \begin{ex}
 	When $M$ is a point, we recover the cohomology of Lie $\infty$-algebras studied in the previous section. When $E_{-i}=0$ for $i \neq 1$, we recover the usual Lie algebroid cohomology (see Chapter IV in Kirill Mackenzie \cite{Kirill05}).
 \end{ex}

Proposition \ref{prop:dual2} is now considered as a classical result \cite{V05,V10,Bonavolonta}, we insert a proof in order to fix sign conventions and notations.

\begin{proof}[Proof of Proposition \ref{prop:dual2}]
	
	Any vector field $Q\in{\mathfrak X}(E)$ is a derivation of $\mathcal{E}$ and
	can be decomposed as a sum
	$$
	Q=\sum_{k\in\Zz} Q^{(k)}
	$$
	where each $Q^{(k)}$  is an arity $k$ map.
	For instance, each homogeneous section of $E$, $e\in\Gamma(E)$, defines a canonical $\mathcal{O}$-linear vector field $i_{e}\in{\mathfrak X}(E)$ of arity $-1$ and degree $|e|$, defined by
	$$
	i_e(\al)=  \eval{\al,e}, \qquad \al\in \E^1=\Gamma(E^*).
	$$

The brackets and the anchor  of the Lie $\infty$-algebroid   corresponding to   a vector field $Q\in{\mathfrak X}({E})$ of degree $+1$ are given by:
 \begin{align*}
  \brr{Q,i_{e}}^{(0)}(f)=&-\rho(e)\brr{f}, \quad\qquad  f\in\mathcal{O}, \, e\in \Gamma(E)\\
 \brr{\ldots\brr{ \brr{Q,i_{e_1}},i_{e_2} },\ldots,i_{e_k}}^{(-1)}
 =&(-1)^{k+1}i_{\set{e_1,e_{2},\ldots, e_k}_k}, \qquad  e_1,\ldots,e_k\in\Gamma(E),
 \end{align*}
 where $\brr{.,.}$ stands for the graded Lie bracket on ${\mathfrak X}(E)$:
 $$
 \brr{Q,P}=Q\smalcirc P - (-1)^{|Q||P|} P\smalcirc Q.
$$

For instance, for $\al\in \Gamma(E^*)$, $e,e_1,\ldots,e_n\in\Gamma(E)$, $f\in\mathcal{O}$,
\begin{align*}
\eval{Q^{(0)}(\al),e}&=(-1)^{|\al|}\eval{\al, l_1(e)},\\
\eval{Q^{(1)}(f), e}&=-\rho(e)(f),\\
\eval{Q^{(1)}(\al), e_1\odot e_2}&= (-1)^{|\al|}\eval{\al,\set{e_1,e_2}_2}-(-1)^{|e_2|}\rho(e_1)(\eval{\al,e_2}) - \rho(e_2)(\eval{\al,e_1})
\end{align*}
and, for each $n\geq 3$, the map $Q^{(n-1)}:\E^1\to \E^n$ is the ``dual" of the $n$-ary bracket~$l_n$:
$$
Q^{(n-1)}(\al)(e_1\odot e_2\odot\ldots\odot e_n)=(-1)^{|\al|} \eval{\al, l_n(e_1, e_2\ldots,  e_n)}.
$$
With this correspondence, it is routine to check that $Q \smalcirc Q=0$ if and only if $(\set{l_k}_{k \geq 1}, \rho)$ is a Lie $\infty $-algebroid structure.
\end{proof}

Lie $\infty$-algebroid morphisms may be defined using the duality developed in Proposition \ref{prop:morphisms}.

\begin{defn}
\label{def:LIe:infty:morphism:oid}
A \textbf{Lie $\infty$-algebroid morphism} $\Phi: (E,Q\equiv(\set{l_k}_{k\geq 1},\rho))\to (E',Q'\equiv(\set{l_k'}_{k\geq 1},\rho'))$  is a degree $0$ algebra morphism $\Phi^*:\mathcal{E'}\to \mathcal{E}$ such that
$$
\Phi^*\smalcirc Q'=Q\smalcirc \Phi^*.
$$
When $\Phi^* $ is $\mathcal O $-linear map, we shall say that the Lie $\infty$-algebroid morphism is \textbf{over the identity of $M$}.
\end{defn}

In particular, a Lie $\infty$-algebroid {morphism} $\Phi$ as in Definition \ref{def:LIe:infty:morphism:oid} induces a graded commutative algebra morphism:
 \begin{equation}\label{eq:HPhi}  H(\Phi) : H(E',Q') \mapsto H (E,Q).
 \end{equation}

\begin{remark}
\normalfont
Lie $ \infty$-algebroid morphisms over the identity of the base manifold $M$ may be defined as in Definition \ref{defn:Lie:infty:morphism} as a family of degree $0$ vector bundle morphisms $\Phi_k \colon S^k (E) \to F $, $k\geq 1$, satisfying several obvious conditions. This description is much more involved for general $ \infty$-algebroid morphisms, as it is for Lie algebroids (see Part I, chapter 4 of \cite{Kirill05}). We will need this point of view in the next section.
\end{remark}

\begin{convention}
    From now on, all Lie $\infty$-algebroids morphisms that we shall consider will be over the identity of the base manifold $M$  and this assumption will be implicit.
\end{convention}

\begin{remark} \label{rem:arity:zero:infty:morphism}
\normalfont
There is another map at cohomology level that should not be confused with the morphism of Equation \eqref{eq:HPhi}.
The arity $0$ component  of a Lie $\infty $-algebroid morphism $\Phi $ induces
a chain map $\phi_1:E_i\to E_i'$, $i\geq 1$, between their linear parts:
\[\xymatrix{
\ldots\ar[r]^{l_1} & E_{-3} \ar[r]^{l_1 }\ar[d]^{\phi_1}& E_{-2} \ar[d]^{\phi_1}\ar[r]^{l_1 }&  E_{-1} \ar[d]^{\phi_1}\ar[r]^{-\rho}& TM \ar[d]^{{\rm{id}}}.\\
\ldots\ar[r]_{l_1'}& E'_{-3} \ar[r]_{l_1'}& E'_{-2}\ar[r]_{l_1'}&  E'_{-1}\ar[r]_{-\rho'}& TM'
}
 \]
 The arity $0$ component of $\Phi $ is in fact entirely determined by this chain map.
\end{remark}

Homotopies of  Lie $\infty$-algebroid   defined exactly as  for Lie $\infty$-algebras (see Definition \ref{def:HomotopyMorphisms}) as being graded differential algebra morphisms:
 $$ \Xi^* \colon (\mathcal E',Q') \to (\mathcal E \hat{\otimes} \Omega^\bullet([0,1]), Q \otimes {\rm id} + {\rm id} \otimes d_{dR})$$
 with the additional condition that $\Xi^* $ has to be over the identity of $M$, i.e. has to be $\mathcal O $-linear. Also, $\hat{\otimes}$ refers to the completion that identifies $\mathcal E \hat{\otimes} \Omega^\bullet([0,1])$ with sections of the graded symmetric algebra of $E^* \oplus T^* [0,1] $ seen as a vector bundle over $M \times [0,1] $. Equivalently, this means that $\Xi_t^*,H_t^*$ are  $\mathcal O $-linear for all $t \in [0,1]$, see \cite{LLS} Section 3.4.3. 
The previous reference provides a proof of the following fact:
if two Lie $\infty$-algebroid morphisms $\Psi$ and $\Phi$ from $(E,Q\equiv(\set{l_k}_{k\geq 1},\rho))$ to $(E',Q'\equiv(\set{l_k'}_{k\geq 1},\rho'))$ are  homotopic, then there exists a $\mathcal{O}$-linear map $\mathcal H:\mathcal{E}'\to \mathcal{E}$ of degree $-1$ such that
\begin{equation}\label{def:homotopic:morphism}
\Psi^*-\Phi^*=Q \smalcirc \mathcal H + \mathcal H\smalcirc Q'.
\end{equation}
In turn, Equation (\ref{def:homotopic:morphism}) implies the following:
\begin{prop}
	\label{eq:homotopiesCohomo}
Two homotopic Lie $\infty$-algebroid morphisms induce the same graded algebra morphism $H^\bullet(E',Q') \to H^\bullet(E,Q) $ in cohomology.
In particular, a homotopy equivalence between two Lie $\infty$-algebroids induces an  isomorphism of their respective cohomologies.
\end{prop}






\subsection{Representations of Lie $\infty$-algebroids}

For Lie algebroids, representations are better seen as being flat connections, a point of view used by Kirill Mackenzie (see, e.g., \cite{YKSKM}).  The same holds for Lie $ \infty$-algebroids \cite{SaSc}.

Let $(E,Q  \equiv(\set{l_k}_{k\geq 1},\rho))$ be a Lie ${\infty}$-algebroid over $M$ and $V=\oplus_{i\in\Zz} V_i$ a graded vector bundle over the same manifold. In the following, we will consider that the anchor $\rho:E\to TM$ is extended by $0$ on   $ \oplus_{i \geq 2} E_{-i} $.


\begin{defn}
\label{def:LieInftyOidConnection}
A {\textbf{Lie $\infty$-algebroid connection}} of $(E,Q\equiv(\set{l_k}_{k\geq 1},\rho))$ on $V$ is a map $\nabla: \Gamma(S(E))\times \Gamma(V)\to \Gamma(V)$ given by a family of degree $+1$ maps $(\nabla^{(k)})_{k\geq 0} $ , 
called {\textbf{Taylor coefficients}},
	$$
	\begin{array}{rrcl}
	\nabla^{(k)}:&\Gamma(S^{k}(E))\times \Gamma(V)&\to& \Gamma(V)\\
	& ( e_1 \odot \cdots \odot e_k  ,  v )& \mapsto & \nabla^{(k)}_{e_1 \odot \cdots \odot e_k} v   \\
	\end{array}
	$$
satisfying the following axioms: 

\begin{enumerate}
     \item $\nabla^{(k)}_{ f e_1 \odot \cdots \odot e_k} v= f\, \nabla^{(k)}_{ e_1 \odot \cdots \odot e_k}v$,
    \item For $\ds k\neq 1$, $ \nabla^{(k)}_{  e_1 \odot \cdots \odot e_k} (fv)= f\, \nabla^{(k)}_{ e_1 \odot \cdots \odot e_k}v$,
      \item $\ds  \nabla^{(1)}_{ e_1 } (fv) = f \,\nabla^{(1)}_{ e_1 } v  + \rho(e_1) [f ] \, v  $,
\end{enumerate}
for every $f\in \mathcal{O}$, $v\in\Gamma(V)$ and $e_1,\ldots, e_k\in \Gamma(E)$, $k\geq 1.$ 
\end{defn}

\begin{remark} \normalfont
In particular, $\nabla^{(0)} : V \to V$ is a vector bundle morphism of degree $+1$. 
\end{remark}

\begin{remark}
\normalfont
    Let us extend $\rho\colon \Gamma(E_{-1}) \to \mathfrak X (M)$ to  $\Gamma(S(E))$ by
     imposing it to be $0$ on $ (\oplus_{i \geq 2} E_{-i})\oplus S^{k\geq 2}E $. Then Axioms (1), (2) and (3) in Definition  \ref{def:LieInftyOidConnection} read 
     $$ \nabla^{(k)}_{  g e_1 \odot \cdots \odot e_k} (fv)= f g\,\nabla^{(k)}_{ e_1 \odot \cdots \odot e_k}v + g\,\rho(e_1,\ldots,e_k)[f]\, v,$$ for all  $ e_1 , \cdots , e_k\in\Gamma(E)$, $k\geq 1$, $ v \in \Gamma(V)$ and  $f,g \in \mathcal{O}$.
\end{remark}


Let us give a dual description of a Lie $\infty$-algebroid connection
of $(E,Q\equiv(\set{l_k}_{k\geq 1},\rho))$ on $V$.
Sections of the tensor product $S(E^*)\otimes V$ come with a natural left $\Gamma(S(E^*))$-module structure given by $$F \cdot (G \otimes v):= (F \odot G) \otimes v,\qquad F,G \in \Gamma(S(E^*)), v\in \Gamma(V).$$ Notice that  we identify $\mathcal O$ and $\Gamma(S^0(E^*))$.

\begin{prop} 
\label{prop:1to1}
\label{Block}
Let $(E,Q\equiv(\set{l_k}_{k\geq 1},\rho))$  be a Lie ${\infty}$-algebroid and $V$ a graded vector bundle over $M$.
There is a one-to-one correspondence between connections of $(E,Q\equiv(\set{l_k}_{k\geq 1},\rho)) $ on $V$ and degree $+1$ operator
\begin{equation*}
\mathcal{D} : \Gamma(S(E^* )\otimes V^*)\longrightarrow \Gamma(S(E^*)\otimes V^*)
\end{equation*}
such that,  for all $\ds \xi\in \Gamma(S(E^*))$ and $\omega\in\Gamma(S(E^* )\otimes V)$,
\begin{equation*}
 \mathcal{D}(\xi {\cdot} \omega)=Q(\xi) {\cdot} \omega + (-1)^{|\xi|}\xi {\cdot} \mathcal{D}\omega.
\end{equation*}
The correspondence goes as follows:
$$  \langle \mathcal D  \alpha  , e_1  \odot \cdots \odot e_k \otimes v \rangle    =  (-1)^{|\alpha|}  \langle   \alpha ,  \nabla^{(k)}_{e_1  \odot \cdots \odot e_k  }  v \rangle +  (-1)^{|\alpha|+1} \rho(e_1  \odot \cdots \odot e_k  ) \left[ \langle \alpha, v \rangle\right] $$
 with the understanding that $\rho(e_1  \odot \cdots \odot e_k  ) $ is zero unless $k=1 $ and $e_1 \in \Gamma(E_{-1})$. Here, $v \in \Gamma(V), \alpha \in \Gamma(V^*), e_1  \odot \cdots \odot e_k  \in \Gamma( S^k (E))  $, $k\geq 1$.  
\end{prop}

We say that a Lie $\infty $-algebroid connection is  {\textbf{flat}}  when $\mathcal D^2 =0$, with $\mathcal D $ as in the above proposition. Alternatively, flat connections are called {\textbf{representations}} of $(E,Q\equiv(\set{l_k}_{k\geq 1},\rho))$ on $V$. This obviously extends Lie algebroid representations in Kirill MacKenzie's \cite{Kirill05}, and representations up to homotopy of Lie algebroids \cite{AC}. 
 Kirill MacKenzie studied those representations in the frameworks of double Lie algebroids, see \cite{GJM}-\cite{Kirill98}.

Let us spell out the meaning of flatness in terms of Taylor coefficients of the connection.  The data in Definition \ref{def:LieInftyOidConnection} define a flat connection if and only if:
\begin{enumerate}
    \item \label{itm:first}$\nabla^{(0)} : V \to V$ makes $V$ a complex, i.e. squares to zero. From now on, we will denote $\nabla^{(0)}  $ by $D_V$ to emphasize on the fact that it is a differential.
    \item \label{itm:second} For all $e \in \Gamma(E)$, we have $${\nabla^{(1)}_{l_1(e)}=-D_V \nabla^{(1)}_e-(-1)^{|e|}\nabla_e^{(1)}D_V=\left\{D_V, \nabla^{(1)}_e\right\} }.$$
    In particular, the symmetric graded commutator $\left\{D_V, \nabla^{(1)}_e\right\}$ is zero for all $e \in \Gamma(E_{-1})$.
    \item \label{itm:third}For all $e_1, e_2 \in \Gamma(E)$, we have:
     $$  \left\{\nabla^{(1)}_{e_1}, \nabla^{(1)}_{e_2} \right\} - \nabla_{\ell_2 (e_1, e_2)}^{(1)}  = {-}\left\{D_V  ,\nabla_{e_1\odot e_2}^{(2)}  \right\} + \nabla^{(2)}_{\ell_1 (e_1) \odot e_2}  + (-1)^{|e_{1}|}\nabla^{(2)}_{e_1 \odot \ell_1(e_2)} $$
     \item More generally, for all  $e_1, \dots, e_n  \in \Gamma(E)$
     \begin{align}
         \label{eq:FlatMeans2}
         &\sum_{\begin{array}{c} \scriptstyle{i=1}\\ \scriptstyle{\sigma\in Sh(i,n-i)}\end{array}}^{\scriptstyle{n}}\!\!\!\!\!\!\!\!\!\! \varepsilon(\sigma)\nabla^{(n-i+1)}_{l_i\left(e_{\sigma(1)}, \ldots, {e_{\sigma(i)}}\right) \odot {e_{\sigma(i+1)}} \odot \cdots \odot {e_{\sigma(n)}}}= \\
&=   \left\{ D_V ,\nabla^{(n)}_{e_1 \odot \cdots \odot e_n}\right\} + \frac{1}{2} \!\!\!\!\!\!\!\!\!\!\sum_{\begin{array}{c} \scriptstyle{j=1}\\ \scriptstyle{\sigma\in Sh(j,n-j)}\end{array}}^{\scriptstyle{n-1}} \!\!\!\!\!\!\!\!\!\!\varepsilon(\sigma)\set{\nabla^{(j)}_{{e_{\sigma(1)}} \odot \cdots \odot {e_{\sigma(j)}}} , \nabla^{(n-j)}_{{e_{\sigma(j+1)}} \odot \ldots \odot {e_{\sigma(n)}}} }.
\nonumber 
   %
  %
 \end{align}
\end{enumerate}
Equation (\ref{eq:FlatMeans2})  is an obvious direct extension of the equivalent formula (\ref{eq:def:representation}) for Lie $\infty $-algebras. All items (1), (2) and (3) are in fact special cases of (\ref{eq:FlatMeans2}).

\begin{defn} \label{def:cohomologyRepresentation}
For a representation $V $ of a Lie $\infty $-algebroid $(E,Q\equiv(\set{l_k}_{k\geq 1},\rho))$, we call $ $ {\textbf{Chevalley-Eilenberg differential}} the operator $\mathcal D $ defined in the Proposition \ref{prop:1to1} and  \textbf{Lie $\infty $-algebroid cohomology}  the cohomology of the complex 
$(\Gamma(S(E^*) \otimes V^*)  , \mathcal D) $.
\end{defn}

Also, Proposition \ref{prop:1to1} allows to give a clear definition of morphisms of representations.

\begin{defn}
A \textbf{morphism} $\Psi:(V,\mathcal{D}_V)\rightarrow (W,\mathcal{D}_W)$ \textbf{between two representations} 
 $V$ and $W$  of $(E,Q\equiv(\set{l_k}_{k\geq 1},\rho))$, with respective Chevalley-Eilenberg differentials $\mathcal D_V,\mathcal D_W $,
is a  degree zero $\Gamma(S(E^*))$-linear map
$$
\Psi:\Gamma(S(E^*)\otimes W^*) \to \Gamma(S(E^*)\otimes V^*),
$$
such that $\ds \Psi\smalcirc \mathcal{D}_W= \mathcal{D}_V\smalcirc \Psi.$
\end{defn}

We now extend to graded vector bundles Kirill Mackenzie's construction of the Atiyah Lie algebroid of a vector bundle (Part I, Section 3 of \cite{Kirill05}).

\begin{defn}
    Let $(V,D_V) $ be a complex, indexed by $\mathbb Z $, of vector bundles over $M$. Define a  graded vector bundle $(A_{-i})_{i \in \mathbb Z} $ over $M$ by:
    \begin{enumerate}
        \item $ A_{-1} := {\mathrm{CDO}}(V) $ is the Atiyah algebroid of the graded vector bundle $V$, i.e. its sections are degree $0$ maps 
        $ \delta : \Gamma(V) \to \Gamma(V)$ such that there exists a unique vector field $\rho(\delta) $ on $M$ satisfying:
        $$ \delta (fv) = f \delta (v) + \rho(\delta)[f] \, v, \quad f \in \mathcal O, v \in \Gamma(V).$$
        \item for $i \geq 2$, $A_{-i} :=   {\mathrm{End}}_{-i+1} (V, V) $ is the vector bundle of vector bundle endomorphisms of degree $-i+1 $. 
    \end{enumerate}
    This space comes equipped with a natural differential graded Lie algebroid structure when equipped with a symmetric graded commutator as a bracket, and the commutator with $D_V $ as a differential. The anchor is the map $\delta \mapsto \rho (\delta) $ of the first item.  We call it the  \textbf{Atiyah differential graded Lie algebroid of $(V,D_V)$}.
\end{defn}

\begin{remark}
\normalfont
The  Atiyah differential graded Lie algebroid is not negatively graded, and cannot be restricted to its negatively graded part. 
\end{remark}

For Lie algebroids,  representations become morphisms valued in the Atiyah Lie algebroids \cite{Kirill05}. Similarly:

\begin{prop}
Let $(E,Q\equiv(\set{l_k}_{k\geq 1},\rho))$  be a Lie $\infty $-algebroid and $(V,D_V)$ a complex of vector bundles over $M$.
There is one-to-one correspondence between flat connections $(\nabla^{(k)})_{k\geq 0}$ of a Lie $\infty$-algebroid $(E,Q\equiv(\set{l_k}_{k\geq 1},\rho))$ on $V$ with $D_V=\nabla^{(0)} $ and Lie $\infty $-algebroid morphisms from $(E,Q\equiv(\set{l_k}_{k\geq 1},\rho))$ to the Atiyah differential graded Lie $\infty $-algebroid of the complex $(V,D_V)$.
\end{prop}
\begin{proof}
    This is precisely the content of Equation \eqref{eq:FlatMeans2}. The $k$-th Taylor coefficient of the Lie $\infty$-morphism is precisely $\nabla^{(k)} $, for $k \geq 1$.
\end{proof}

Also, representations of Lie $\infty $-algebroids may be described as semi-direct products.

\begin{prop} A representation
$(\nabla^{(n)})_{n \geq 1} $ of a Lie $\infty$-algebroid $(E,Q\equiv(\set{l_k}_{k\geq 1},\rho))$ over $V$  endows the graded vector bundle $E\oplus V$ with a structure of Lie $\infty$-algebroid $(\ell_k)_{k \geq 1} $ when equipped with  the brackets:
 \begin{align*}
 \ell_k(e_1 , \dots, e_{k}) &=  l_k(e_1 , \dots, e_{k})
     \\
 \ell_k(e_1 , \dots, {e_{k-1}}, v) & =   \nabla^{(k-1)}_{e_1 \odot \cdots \odot e_{k-1}} v
 \end{align*} 
for all $e_1, \dots, e_k \in \Gamma(E), v \in \Gamma(V) $. All other brackets are zero and the anchor is  $(e,v) \mapsto \rho(e)$, for all $(e,v) \in \Gamma(E_{-1} \oplus V_{-1}) $.
\end{prop}
\begin{proof}
A direct computation shows that flatness is equivalent to higher Jacobi identities. All other axioms are obvious.
Alternatively, it also follows from Proposition \ref{prop:1to1}.
\end{proof}

A representation of a Lie algebroid on a vector bundle $V$ induces a representation on $V^*$. The same applies for a Lie $  \infty $-algebroid. 

Let $(E,Q\equiv(\set{l_k}_{k\geq 1},\rho))$  be a Lie $\infty $-algebroid and  $(V,D_V)$ a complex of vector bundles over $M$.
 Any connection $(\nabla^{(n)})_{n \geq 0} $ with $\nabla^{(0)}=D_V $ admits  
 a \textbf{dual connection
 }
 $(^*\nabla^{(n)})_{n \geq 0} $  on the dual graded vector bundle $V^*$. In terms of Taylor coefficients, it is defined by $^*\nabla^{(0)}=-D_V^* $
and
 $$ 
  \langle  ^*\nabla^{(n)}_{ e_1 \odot \cdots \odot e_n}   \alpha, v  \rangle +(-1)^{|\alpha|(|e_1|+ \cdots |e_n|+1)}
 \langle \alpha ,\nabla^{(n)}_{ e_1 \odot \cdots \odot e_n}  v \rangle  = \rho( e_1 \odot \cdots \odot e_n ) \langle \alpha, v\rangle  $$
 for all $\alpha \in \Gamma(V^*), v \in \Gamma(V), e_1 , \cdots , e_n \in \Gamma(E) $.
 
\begin{prop}
 The dual of the dual of a connection $\nabla $ is mapped to $\nabla $ under the canonical isomorphism $(V^*)^* \simeq V $. 
A Lie $\infty $-algebroid connection is flat if and only if its dual connection is flat.
\end{prop}

\subsection{Adjoint and coadjoint representations}

\label{sec:adjointcoadjoint}

Adjoint and coadjoint representations of Lie $\infty $-algebroids are not as easy to define as those of Lie $\infty $-algebras: there is no clear equivalent of Equations \eqref{eq:adjoint:representation:algebra}. Also, adjoint and coadjoint representations of Lie algebroids are not so easy to define, and depend on the choice of a connection \cite{AC}.
For Lie $ \infty$-algebroids, they can be defined upon a choice of connections on $E$, as we will now see.
Let $(E,Q\equiv(\set{l_k}_{k\geq 1},\rho))$ be a Lie $\infty$-algebroid over $M$.
Recall that $\mathcal E:= \Gamma(S(E^*))$. 
Let $ {\mathfrak X}(E)= \oplus_{i \in {\mathbb Z}} {\mathfrak X}_i(E) $ be the graded Lie algebra of vector fields on the graded manifold $E$ (i.e. graded derivations of $\underline{\mathcal E}$),
then $$ {\rm ad}_Q : {P} \mapsto [Q,P ]=Q\smalcirc P -(-1)^{|P|} P\smalcirc Q $$ is a degree $+1$ operator on $ {\mathfrak X}(E)$ squaring to zero. It satisfies the graded Leibniz identity:
\begin{equation}
\label{eq:Leibniz}
{\rm ad}_Q (F.P)= Q[F] \, {P} + F. {\rm ad}_Q {P}.
\end{equation}
for all $F \in {\mathcal E}$, $P \in  {\mathfrak X}(E)$.

\begin{defn}
We call {\textbf{ abstract coadjoint complex}} the complex  $ ({\mathfrak X}(E)= \oplus_{i \in {\mathbb Z}} {\mathfrak X}_i(E) , {\rm ad}_Q)$.
\end{defn}

Let us choose a family $\nabla=(\nabla^i)_{i \geq 1}$ of vector bundle connections\footnote{Here $\nabla^i$ is a $TM$-connection on $E_{-i}$, and must not be confused with the $i+1$-ary operation~$\nabla^{(i)}$}  on each one of the vector bundles $(E_{-i})_{i \geq 1}$.
These connections allow to map a vector field $ X  $ on $M$ to the unique degree $0$ derivation of $ {\mathcal E}$ whose restriction to $\Gamma(E_{-i}^*)$ is the dual of the covariant derivative $\nabla^i_X$.
Also, $e \in \Gamma(E_{-i})$ maps to the unique degree $-i$ derivation of $ {\mathcal E}$ whose restriction to $\Gamma(E_{-j}^*)$
is $0$ for $ i \neq 0$ and whose restriction to $\Gamma(E_{-i}^*) $ is the contraction with $e$. Altogether, these mappings give an identification of $ {\mathcal E}$-modules:
$$
{\mathcal E} \otimes_{\mathcal O} \Gamma \left(E\oplus TM \right) \simeq
{\mathfrak X}(E).
$$

Under this identification, the abstract coadjoint complex becomes a degree $+1$ linear map squaring to zero
$$\ad_Q:{\mathcal E} \otimes_{\mathcal O} \Gamma \left(E\oplus TM \right)\to{\mathcal E} \otimes_{\mathcal O} \Gamma \left(E\oplus TM \right).$$
 Equation (\ref{eq:Leibniz}) and Proposition \ref{prop:1to1} imply that it is a representation of $(E,Q\equiv(\set{l_k}_{k\geq 1},\rho))$ on a graded vector bundle which is $T^*M$ in degree $0$ and $(E_{-i})^*$ in degree $+i$.
	We call it the {\textbf{coadjoint representation of $(E,Q\equiv(\set{l_k}_{k\geq 1},\rho))$}} associated to the connection $\nabla$.
	We call {\textbf{adjoint representation of $(E,Q\equiv(\set{l_k}_{k\geq 1},\rho))$}} its dual representation: it is by construction a representation of  $(E,Q\equiv(\set{l_k}_{k\geq 1},\rho))$ on its linear part:
	 $$ \cdots \longrightarrow  E_{-2} \stackrel{l_1}{\longrightarrow} E_{-1}   \stackrel{-\rho}{\longrightarrow}  TM. $$



\begin{remark}\normalfont
\label{rmk:isomorphicrepresentation}
Since the coadjoint complex associated to the connection $\nabla=(\nabla^i)_{i \geq 1}$ is isomorphic to the  abstract coadjoint complex (as differential graded $\underline{\mathcal E}$-modules), the coadjoint representations associated to two different choices of connections are strictly isomorphic (in a canonical manner). The same holds for the adjoint representation.
\end{remark}



Proposition \ref{prop:adjoint:connection} provides an explicit  description of the adjoint representation, where:
\begin{enumerate}
	\item For the sake of simplicity, we now denote all the connections $\nabla^i$ by the same symbol $\nabla$,
	\item  $\brr{\nabla_X,l_k}$ stands for the graded commutator of coderivations of $\Gamma(S(E))$, namely:
	\begin{align*}
\brr{\nabla_X,l_k}&(e_1,\ldots, e_k)  \\&= \nabla_Xl_k(e_1,\ldots, e_k) - l_k(\nabla_X e_1,\ldots,e_k) - \ldots - l_k(e_1,\ldots, \nabla_X e_k),	\end{align*}
	for all $e_1,\ldots, e_k\in\Gamma(E)$, $X \in \mathfrak X(M)$.
\end{enumerate}
We hope that it can be of use for future references, since it seems that it has never been written down explicitly.

\begin{prop}\label{prop:adjoint:connection} Let $(E,Q\equiv(\set{l_k}_{k\geq 1},\rho))$ be  a Lie $\infty$-algebroid.   The Taylor coefficients of the adjoint representation $\ad^{\nabla}$  induced by a  connection $\nabla$ on $E$ are given by the differential
	 \begin{equation}
	 \label{eq:diff}
	   \xymatrix{ \cdots \ar^{l_1}[r] & E_{-2} \ar^{l_1}[r]& E_{-1} \ar^{-\rho}[r]& TM}  . 
	 \end{equation}
	and the following family of degree $+1$ maps $ (\ad^{\nabla\, (k)})_{k \geq 1}$:
\begin{equation*} \left\{
\begin{array}{rcl}
\ad^{\nabla\, (1)}_{e'}(e)&=&  \set{e',e}_2 -(-1)^{|e'|} \nabla_{\rho(e)}e',
\\
\ad^{\nabla\, (1)}_{e'}(X)&=& \brr{\nabla_X,l_1}(e')+ \brr{\rho(e'),X} + \rho(\nabla_X e')\\
\ad^{\nabla\, (2)}_{e_1\odot e_2}(X)&=& \brr{\nabla_X,l_2}(e_1, e_2) \\
& & \, \, \, + \nabla_{\brr{\rho(e_1),X} + \rho(\nabla_X e_1)}e_2 + (-1)^{|e_1|} \nabla_{\brr{\rho(e_2),X} + \rho(\nabla_X e_2)}e_1\\
\ad^{\nabla\, (k)}_{e_1\odot\ldots\odot e_k}(e)&=&  \set{e_1,\ldots,e_k,e}_{k+1},\quad k\geq 2,
\\
\ad^{\nabla\, (k)}_{e_1\odot \ldots\odot e_k}(X)&=& \brr{\nabla_X,l_k}(e_1,\ldots, e_k),\quad k\geq 3,
\end{array}\right.
\end{equation*}
where $e,e',e_1,\ldots e_k\in\Gamma(E)$ and $X\in {\mathfrak X}(M)$.
\end{prop}


\begin{ex} For a Lie algebroid $E=E_{-1}$,  the adjoint representation is the adjoint representation (up to homotopy)  constructed in \cite{AC}.
\end{ex}

\subsection{Berezinian bundle}

%

We define the {\textbf{Berezinian bundle}}  (or ``Berezinian" for short) of a finite dimensional graded bundle $V = \oplus_{i \in {\mathbb Z}} V_i $ over $M$ to be the line bundle 
  $$ {\mathrm{Ber}}(V) :=  \cdots \otimes  \wedge^{\top} V_2^* \otimes  \wedge^{\top} V_1 \otimes \wedge^{\top} V_0^* \otimes \wedge^{\top} V_1 \otimes \wedge^{\top} V_{2}^* \otimes \cdots $$
 with the understanding that $ \wedge^{\top} E$ stands for   $ \wedge^{{\rm{ dim}}(E)} E$ for any finite dimensional vector space $E$. The infinite tensor product above is indeed finite, hence well-defined. 
 We consider the Berezinian bundle to be
 of degree $0$ (in fact, any degree would work).
 The following lemma is a well-known fact of homotopy theory. 

\begin{lem}
\label{lem:berezinian}
  A homotopy equivalence between finite dimensional complexes of vector bundles induces a canonical isomorphism of their Berezinian bundles.
\end{lem}
  
  Let us describe sections of this bundle. 
  Given a family $\mu:=(\mu_{i})_{i \in \mathbb Z} $, with $\mu_{i}$ a nowhere vanishing section of $\wedge^{\top} V_{i} $ (defined in some open subset $\mathcal U \subset M$), we denote by $\mu_{2i}^* $
  the dual nowhere vanishing section of $\wedge^{\top} V_{2i}^*$.
  %
  The tensor product:
   \begin{equation}\label{eq:defmu}  \mu :=  \cdots   \otimes \mu_{-2}^* \otimes   \mu_{-1}  \otimes  \mu_0^* \otimes \mu_1 \otimes  \mu_2^* \cdots
   \end{equation}
  is a nowhere vanishing section $ {\mathrm{Ber}}(V) $ on $\mathcal U \subset M $.

  \begin{thm}
  \label{thm:CharacteristicClass}
  For any representation $(\nabla^{(n)})_{n \geq 0}$ of a Lie $ \infty$-algebroid $(E,Q\equiv(\set{l_k}_{k\geq 1},\rho))$ on a finite dimensional graded vector bundle $V$, the  Berezinian bundle
  ${\mathrm{Ber}}(V)$ comes equipped with a natural representation structure whose unique non-vanishing Taylor coefficient is, for all $e \in \Gamma(E_{-1})$, given by: 
  \begin{equation}\label{eq:nablaBer} \nabla^{\mathrm{Ber}}_e   \mu = \left(  \sum_{i \in \mathbb Z} (-1)^i {\rm{div}}_{\mu_i} (\nabla^{(1)}_e )  \right)  \,  \mu,  \end{equation}
for every $\mu $ as in \eqref{eq:defmu}, where  ${\rm{div}}_{\mu_i}$ is defined by the relation
   $$ \sum_{k=1}^d  v_1 \wedge \dots \wedge \left( \nabla^{(1)}_e v_k \right) \wedge \dots \wedge v_d   =  {\rm{div}}_{\mu_i}\left(\nabla^{(1)}_e \right)  \, \, \, v_1 \wedge \dots \wedge v_d   $$
   for every nowhere vanishing local sections $ v_1 , \dots, v_d$ of  $V_i $ such that $ v_1 \wedge \dots \wedge  v_d = \mu_i$ is a section of $\wedge^{\top} V_i$.
  \end{thm}
    \begin{proof}
For each trivialization  $(v_i)_{i \in I}$  of $\Gamma(V)$ by homogeneous sections of $V$, consider $(\al_i)_{i\in I}$ the dual trivialization of $\Gamma(V^*)$.  The \textbf{connection functions} $\omega_i^j\in \mathcal{E}$ of $\nabla$ are defined by
in terms of the Chevalley-Eilenberg differential $\mathcal{D}$ by:
$$
\mathcal{D}\al_i=\sum_{j \in I} (-1)^{|\al_i|(|\al_j|+1)}\omega_i^j\otimes \al_j .
$$
Notice that, since $\mathcal{D}$ has degree $+1$, each connection function has degree $|\omega_i^j|=1+|v_i|-|v_j|$.
In particular, $\omega_i^i $ has arity $1$ (i.e. is a section of $(E_{-1})^*$ for all index $i$), and is defined by:
 \begin{equation}
 \label{eq:omegaii}
 \omega_i^i(e)=\eval{\al_i,\nabla^{(1)}_e v_i},\quad  e\in \Gamma(E_{-1}) .\end{equation}
In particular $\str \omega = \sum_{i\in I} (-1)^{|\al_i|} \omega_i^i $ has arity $+1$  and degree $+1$. 

Let $\mu$ be the section of $\mathrm{Ber}(V)$ defined by the chosen local trivialization of $\Gamma(V)$. Then Equation \eqref{eq:nablaBer} reads:
\begin{equation}
\label{eq:Div=Supertrace}
 \sum_{i \in \mathbb Z} (-1)^i {\rm{div}}_{\mu_i} (\nabla^{(1)}_e ) =\str \omega(e) .   
\end{equation}

Let us consider another (local) trivialization $(v_i')_{i \in I}$ such that $v'=Av$, for some matrix of local smooth functions $A=(A_i^j)_{i,j \in I}$. The coefficient $A_i^j$ of this matrix is zero if $|v_i| \neq |v_j| $, so that $A$ is a block  diagonal matrix with blocks $(A_n)_{n \in \mathbb Z}$.
 A direct computation gives that the connection functions $\omega' $ and $\omega $ for these trivializations are related by (using Einstein's convention):
 $$
 {\omega'_i}^j=(A^{-1})_p^i\omega_k^p\, A_j^k -Q(A_j^k)(A^{-1})_k^i. 
 $$
 Since $A_{i}^{j}=0$ when the degrees of $v_i$ and $v_j$ are different,
 we have, by invariance of the trace:
 $$  \sum_{i \in I \hbox{ s.t. } |v_i|=j} (A^{-1})_{p}^{i}\omega_k^p A_j^k= \sum_{i \in I \hbox{ s.t. } |v_i|=j} \omega_i^i$$
 so that:
 \begin{equation}\label{eq:omegaomega} \str\omega' - \str\omega = -(-1)^{|v_i|}Q(A_{i}^{k})(A^{-1})_k^{i}=\frac{Q[({\rm {Ber}(A)]}}{{\rm{ Ber}}(A)}= Q[{\rm{ln}} (|{\rm{Ber}}(A)|)].
 \end{equation}
Let us use Equations \eqref{eq:Div=Supertrace} and \eqref{eq:omegaomega} to show that \emph{(i)}
$\nabla^{Ber} $ is well-defined and \emph{(ii)} is a Lie $\infty $-algebroid connection which \emph{(iii)} is flat.

For two different trivializations as above,
the associated sections of $\mathrm{Ber}(V)$ are related by
$$
\mu'=\mathrm{Ber}(A)\,\mu,
$$
where $\mathrm{Ber}(A)= \frac{\prod_k {\mathrm{det}} A_{2k+1}}{ \prod_k {\mathrm{det}} A_{2k}}$.
Equation \eqref{eq:omegaomega} implies that 
$$ \str\omega' \, (e) - \str\omega \, (e)   =  \rho(e) [{\rm{ln}} (|{\rm{Ber}}(A)|)],  \quad  e \in \Gamma(E_{-1}),$$
 This implies \emph{(i)}
and proves that $\nabla^{\mathrm{Ber}}_e$ is well-defined. It also implies \emph{(ii)}, 
i.e. that $\nabla^{\mathrm{Ber}}_e$ is  $\mathcal{O} $-linear in $E$ while 
$\nabla^{Ber}_e f \, \mu = f  \nabla^{Ber}_e\mu+ \rho(e)[f]\, \mu,$
for all  $f\in\mathcal{O}$.

To prove \emph{(iii)}, we will use the  flatness of $\mathcal{D}$. For each $i\in I$,
\begin{equation*}
\mathcal{D}^2\al_i=\sum_j (-1)^{|\al_i|(|\al_j|+1)}Q(\omega_i^j)\otimes \al_j + (-1)^{1+(|\al_i|+|\al_j|)|\al_k|}\omega_i^k\odot \omega_{k}^j \otimes \al_j=0,
\end{equation*}
 and, consequently,
 \begin{equation}
 \label{eq:CartanRepr}
 Q(\omega_i^j)=\sum_k (-1)^{(|\al_i|+|\al_j|)|\al_k|+|\al_i|(|\al_j|+1)}\omega_i^k\odot \omega_k^j.
 \end{equation}
 In particular,
 \begin{align*}
Q(\str \omega)&=
 \sum_i(-1)^{|\al_i|} Q(\omega_i^i) \\ &=\sum_{i,k} (-1)^{|\al_i|}\omega_i^k\odot \omega_k^i \hbox{ by Eq. (\ref{eq:CartanRepr})}\\ &
=\sum_{i,k} (-1)^{|\al_i|}(-1)^{(1+|\al_i|-|\al_k|)}\omega_k^i\odot \omega_i^k \hbox{ for degree reason} \\
&=\sum_{i,k} (-1)^{1+|\al_k|}\omega_k^i\odot \omega_i^k=-\sum_k(-1)^{|\al_k|} Q(\omega_k^k)= -Q(\str \omega).
 \end{align*}
 This implies that
 $\str \omega = \sum_{i \in I} (-1)^{|v_i|} \omega_i^i $
  is a $Q$-closed element in ${\mathcal E}_1$.  In view of flatness conditions  (\ref{itm:first})-(\ref{itm:third}) in page \pageref{itm:first}, this is equivalent to $\nabla^{\mathrm{Ber}}$ being a representation:
  \begin{align*}
      \nabla^{\mathrm{Ber}}_{l_1(e)}&=0\\
      \nabla^{\mathrm{Ber}}_{\set{e_1,e_2}}&=\left[\nabla^{\mathrm{Ber}}_{e_1},\nabla^{\mathrm{Ber}}_{e_2}\right],
  \end{align*}
  for $e_1,e_2\in \Gamma(E_{-1})$.

\end{proof}

\section{The modular class of a Lie ${\infty}$-algebroid} \label{LieInftyModular}

\subsection{Some definitions and computations}

  The Berezinian bundle of a representation of a Lie $\infty$-algebroid $(E,Q\equiv(\set{l_k}_{k\geq 1},\rho) )$ has two characteristics: \emph{(i)} it is concentrated in some degree, and \emph{(ii)}  it has rank $1$ in this degree. This implies that only its first Taylor coefficient may be non-zero, and only when applied to a section of $E_{-1} $. Given now a {\textbf{rank-$1$ representation}} $ B \to M$, i.e. a representation that admits these two characteristics, denote by $(e,b) \mapsto \nabla^{B}_e b $ its Taylor coefficient.

From now on, we work in smooth differential geometry, i.e. $\mathcal O = C^\infty(M)$. Assume there exists a nowhere vanishing section $b$ of $B$,  then the section $\omega_b \in \Gamma(E_{-1}^*) = \mathcal E_{+1} $ defined by:
   $$  \nabla^{B}_e b = \omega_b(e) \, b $$
$Q$-closed function on $E$ i.e.:
   $$ \omega_b(l_2(e_1,e_2)) - \rho(e_1) [ \omega_b(e_2)] + \rho(e_2) [\omega_b(e_1)] =0  \hbox{ and } \omega_b( l_1 (\tilde e)) =0, $$
   for any $e_{1},e_{2} \in \Gamma(E_{-1})$, $\tilde e \in \Gamma(E_{-2})$.
   Also, it is easily checked that for every   nowhere vanishing smooth function $\lambda \in \mathcal{O}$:
    $$  \omega_{\lambda b} (e) = \omega_{b} (e) + \frac{\rho(e) [\lambda]}{\lambda}  = \omega_{b} (e) +  \rho(e)  [\mathrm{ln}(|\lambda|) ] $$
    or, equivalently, $ \omega_{\lambda b} =\omega_{ b} +   Q_E [\mathrm{ln}(|\lambda|) ]$. As a consequence, the class of $\omega_b  \in H^1(E,Q) $ does not depend on the choice of $b$. If there is no non-vanishing section because the rank $1$ bundle $B$ is not trivial, then we can consider the representation $B \otimes B $, which is now trivial as a rank $1$ bundle, and consider one-half of the class defined above. We call this class the {\textbf{characteristic class of the rank-$1$ representation~$B$}}. 
    
    \begin{defn} 
        We call {\textbf{characteristic class of a representation}} of a Lie $\infty $-algebroid the characteristic class of its Berezinian bundle (defined in  Theorem \ref{thm:CharacteristicClass}).
%
%
    \end{defn}

Let $(E,Q\equiv(\set{l_k}_{k\geq 1},\rho))$ be a Lie $\infty$-algebroid and $\nabla= (\nabla^i)_{i \geq 1}$ a connection on  $E$. The adjoint representation $\ad^\nabla$  (see Section \ref{sec:adjointcoadjoint}) has a characteristic class.
Since different connections define isomorphic representations (see Remark \ref{rmk:isomorphicrepresentation}), we have:

\begin{prop}
\label{prop:isorepres}
	Let $ (E,Q\equiv(\set{l_k}_{k\geq 1},\rho)) $ be a Lie $\infty$-algebroid.
The characteristic class of the adjoint representation  with respect to a connection $\nabla$ is independent of the chosen connection.
\end{prop}


\noindent
By Proposition \ref{prop:isorepres}, the following definition makes sense.

\begin{defn}
The \textbf{modular class of a Lie $\infty$-algebroid} $(E,Q\equiv(\set{l_k}_{k\geq 1},\rho))$ is the characteristic class of any of its adjoint representations.
When the modular class is zero,  we say that the Lie ${\infty}$-algebroid is {\textbf{unimodular}}.
\end{defn}

Let us give  a concrete description of the modular class.
Let $X_1, \dots,X_n$ and $e_1^{(i)}, \dots,e_{a_i}^{(i)}$ be local trivializations of $TM$
and each one of the vector bundles $E_{-i}$, over some open subset of $M$.

For every $Y \in {\mathfrak X}(M)$, we define ${\rm div}(Y)$ to be the unique function that satisfies:
   $$ [Y,X_1 \wedge \dots  \wedge X_n ]  = {\rm div}(Y) X_1 \wedge \dots  \wedge X_n$$
   and we define ${\rm{div}^{(i)} (e)} $ to be, for all $i \geq 1$ and all $e \in  \Gamma(E_{-1})$,
   the unique function that satisfies:
    $$  \{e, e_1^{(i)} \wedge \dots \wedge e_{a_i}^{(i)} \}_2 = \sum_{k=1}^{a_i} e_1^{(i)} \wedge \dots \wedge  \{e, e_k^{(i)} \}_2 \wedge \dots \wedge e_{a_i}^{(i)} = {\rm{div}^{(i)} (e)} \, \, e_1^{(i)} \wedge \dots \wedge e_{a_i}^{(i)} $$
then

\begin{defn}
\label{def:modularoid}
	Let $(E,Q\equiv(\set{l_k}_{k\geq 1},\rho))$ be a Lie $\infty$-algebroid.
	The {\textbf{modular cocycle}} with respect to  the local trivializations $X_1, \dots,X_n$ and $e_1^{(i)}, \dots,e_{a_i}^{(i)}$ as above is the section $\omega_{\nabla} $ of $E_{-1}^*$ given for all $e \in \Gamma(E_{-1})$ by:
	$$ \omega_\nabla (e) =  {\rm div}(\rho(e)) + \sum_{i \geq 1} (-1)^i {\rm{div}^{(i)} (e)}.  $$
\end{defn}
	
In the expression of the adjoint action of a Lie $\infty$-algebroid, as described in Proposition \ref{prop:adjoint:connection}, the Taylor coefficients (even the one that will appear when computing the action on the Berezinian bundle) have terms involving the $TM$-connection on $E$ that do not appear in the expression of $\omega_\nabla $. The next statement is therefore not obvious, and deserves a careful proof.
	
\begin{prop}
\label{prop:cocycle}
	Any modular cocycle  is a representative of the modular class.
\end{prop} 
\begin{proof} Let $x_1, \dots, x_n$ be local coordinate in an coordinate neighborhood  $U \subset M$, and $e_{1}, \dots, e_r$ be local trivialization of $E$ over $U$ (by homogeneous sections). We denote the dual variables by $\xi^1, \dots,\xi^r $ and by $\mu$ the  section of the Berezinian bundle of  $E\oplus TM$ constructed out the local trivializations $\frac{\partial}{\partial x_1}, \dots, \frac{\partial}{\partial x_n}$ and $e_1\ldots, e_r$.
 
From Theorem \ref{thm:CharacteristicClass} we know  that $\nabla^{Ber}_e\mu=\str \omega(e)\,\mu$, $e\in\Gamma(E_{-1})$ where
\begin{align*}
\str \omega(e)=\sum_{k=1}^n  \eval{d x_k, \ad^{(1)}_{\frac{\partial}{\partial x_k}}(e)}
+\sum_{k=1}^r \eval{\xi^k, \ad^{(1)}_{e_k}(e)}.
\end{align*}

Taking into account the Taylor coefficients of the adjoint representation given in Proposition \ref{prop:adjoint:connection} we have 
\begin{align}
\label{eq:omeganabla}
\str \omega(e) =& \sum_{i=1}^n \eval{dx_i,\brr{\rho(e),\frac{\partial}{\partial x_i}}}
+\sum_{i=1}^n  \eval{d x_i, \rho(\nabla_{\frac{\partial}{\partial x_i}}(e))}\\
\nonumber
&+(-1)^{|e_k|} \sum_{k=1}^r \eval{\xi^k, \set{e_k,e}_2} 
+ (-1)^{|e_k|}\sum_{k=1}^r \eval {\xi^k, \nabla_{\rho(e_k)}(e)} .
\end{align}
Now, consider the functions  $\rho_k^i=\eval{\rho(e_k),dx_i}=\eval{\xi^k,\rho^*(dx_i)}$ and notice that
\begin{align*}
\sum_{k=1}^r \eval {\xi^k, \nabla_{\rho(e_k)}(e)}& =\sum_{k=1}^r\sum_{i=1}^n \eval {\xi^k, \nabla_{\rho_k^i\frac{\partial}{\partial x_i}}(e)}\\
&=\sum_{k=1}^r\sum_{i=1}^n \eval {\rho_k^i\,\xi^k, \nabla_{\frac{\partial}{\partial x_i}}(e)}=\sum_{i=1}^n\eval{(-1)^{|\xi_k|}\rho^* (dx_i),\nabla_{\frac{\partial}{\partial x_i}}(e) } 
\end{align*}
As a consequence, the second and fourth terms in \eqref{eq:omeganabla} add up to zero. Therefore
\begin{align}
\label{eq:twoomegas}
\str \omega(e) &=  \sum_{i=1}^n \eval{dx_i,\brr{\rho(e),\frac{\partial}{\partial x_i}}}+ \sum_{k=1}^r (-1)^{|\xi_k|}\eval{\xi^k, \set{e_k,e}_2}\\ \nonumber
&= {\rm div}(\rho(e)) + \sum_{i \geq 1} (-1)^i {\rm{div}^{(i)} (e)}\\ \nonumber
& = \omega_\nabla (e).
\end{align}
This completes the proof
 
\end{proof}

\begin{remark}
\normalfont
\label{rmk:rhogam}
 Let $x_1, \dots, x_n,e_{1}, \dots, e_r$ be as in the proof of Proposition \ref{prop:cocycle}.
 For every index $i$ such that $e_i $ has degree  $ -1$, we define functions $(\rho_a^i)_{a=1, \dots,n}  $ by:
 $$ \rho (e_i) = \sum_{a=1}^n \rho_a^i \frac{\partial}{ \partial x_a}   $$
  and we define functions $(\Gamma_{i,j}^k)_{1 \leq i,j,k\leq d}$ by
  $$ l_2(e_i,e_j ) = \sum_{k=1}^d \Gamma_{i,j}^k e_k .$$
  (Of course $\Gamma_{i,j}^k = 0$ unless $|e_k|=|e_i|+|e_j|+1 $).
  For every index $i$ such that $e_i$ has degree  $ -1$, Equation \eqref{eq:twoomegas} in the proof of Proposition \ref{prop:cocycle} reads:
  \begin{equation}
  \label{eq:gammarho}
  \omega_\nabla (e_i) = 
    \left(\sum_{a=1}^n \frac{\partial \rho_a^i}{\partial x_a} + \sum_{j=1}^d (-1)^{|e_j|}\Gamma_{i,j}^j \right).
    \end{equation}
\end{remark}

An immediate consequence of Proposition \ref{prop:cocycle} is the following result
(which extends Example \ref{ex:negGraded}).

\begin{cor}
\label{cor:negativelyGradedLieInftyDepends On}
The modular class of a negatively graded Lie $\infty$-algebroid $(E,Q\equiv(\set{l_k}_{k\geq 1},\rho)) $ depends only on $l_1$, $\rho$ and on the restriction of $l_2  $ to $\Gamma(E_{-1}) \otimes \Gamma(E) \to \Gamma(E).$
\end{cor}

Let us conclude this section by explaining the relation with the existing studies of modular classes \cite{AB} and \cite{V07}.

\begin{remark}
\normalfont
In local coordinates, the modular cocycle is the divergence of $Q_E $, so that the modular class as defined as above matches the modular class as defined in \cite{AB,V07}. Let us check this point 
 that generalizes Proposition \ref{prop:divergence}.  Let $x_1, \dots, x_n$ be local coordinate in an coordinate neighborhood  $U \subset M$, and $e_{1}, \dots, e_r$ be local trivialization of $E$ over $U$ (for homogeneous sections). We denote the dual variables by $\xi^1, \dots,\xi^r $. In general, the divergence of a vector field $Q$ of degree $i$ is defined by:
 \begin{equation}
\label{def:div}
 {\mathrm{Div}}(Q) := (-1)^{i+1} \left(\sum_{k=1}^r \frac{\partial }{\partial \xi^k} Q[\xi^k] +  \sum_{k=1}^n \frac{\partial }{\partial x_k} Q[x_k]\right)  
 \end{equation}
For vector fields of degree $ 1$,  for degree reasons, the divergence of the components of $Q $ of arity $\geq +2 $ or $0$ will disappear in the sum on the right-hand side of 
\eqref{def:div}: 
$$
{\mathrm{Div}}(Q) :=  \sum_{k=1  }^r \frac{\partial }{\partial \xi^k} Q^{(1)}[\xi^k] +  \sum_{k=1}^n \frac{\partial }{\partial x_k} Q^{(1)}[x_k]
$$
Define functions $\rho_a^j  $ and $\Gamma_{i,j}^k $ as in Remark \ref{rmk:rhogam}.
Then
 $$ Q^{(1)} = \sum_{i,a} \rho_i^{a} \xi^i  \frac{\partial}{ \partial x_a}  + \sum_{i,j,k}  (-1)^{|\xi^k|}\Gamma_{i,j}^k \xi^j \xi^i  \frac{\partial}{ \partial \xi^k}.$$
 Hence 
 $$ {\mathrm{Div}}(Q) =  \sum_{i =1}^r  \left(\sum_{a=1}^n\frac{\partial \rho_a^i}{\partial x_a} + \sum_{j=1}^r  (-1)^{|\xi^j|}\Gamma_{i,j}^j \right) \xi^i $$
 The sum over $i$ indeed only runs on indices such that $|\xi^i|=1 $.
 This is the definition of the modular cocycle in \cite{AB}.  Equation \eqref{eq:gammarho} means that this definition matches our Definition \ref{def:modularoid}.
\end{remark}

\subsection{The modular class, leaf by leaf description, and homotopy invariance}

 For any negatively graded Lie $ \infty$-algebroid $(E,Q\equiv(\set{l_k}_{k\geq 1},\rho)) $, the image $\rho(\Gamma(E_{-1})) $  of the anchor map is a singular foliation in the sense of \cite{AS,AZ}. In particular, the base manifold $M$ is a disjoint union of leaves. More precisely, any  equivalence class $L $ of the equivalence relation on $M$ defined by:  \emph{``$m_1\sim m_2$ if and only there exists a finite family of vector fields in $\rho(\Gamma(E_{-1}))$ whose successive time-$1$ flows maps $m_1$ to $m_2$"} is a submanifold such that $T_m L = \rho_m (E_{-1}|_m) $, at every $m \in L$. We call these submanifolds {\textbf{leaves of $(E,Q\equiv(\set{l_k}_{k\geq 1},\rho)) $}}.

Let $L$ be such a leaf.
\begin{enumerate}
    \item The restriction of the linear part $(E\oplus TM,l_1,-\rho) $ to  $L$ is a complex of vector bundles
    \begin{equation}\label{eq:leafwisecomplex} \cdots \stackrel{l_1}{\longrightarrow} {\mathfrak i}_L E_{-2} \stackrel{l_1}{\longrightarrow}  {\mathfrak i}_L E_{-1}  \stackrel{-\rho}{\longrightarrow} {\mathfrak i}_L TM \end{equation}
    whose differential has constant rank. 
    \item[] This can be proven as follows. The kernel and image of $ \rho$ have constant rank along $L$ since $\rho( {\mathfrak i}_L E_{-1})=TL $. Let $\sigma \colon TL \to {\mathfrak i}_L E_{-1} $ be a section of the anchor map $ \rho \colon {\mathfrak i}_L E_{-1} \to TL $. Consider the connections on the vector bundles $({\mathfrak i}_L E_{-i})_{i \geq 2} $ 
    and the vector bundle ${\mathrm{Ker}}(\rho) \subset {\mathfrak i}_L E_{-1}$, 
    defined by:
     $$ \nabla^\sigma_u e := l_2 (\sigma(u), e ) \hbox{ for all $ u \in \mathfrak X(L), e \in \Gamma(E_{-k})$ or $ e \in \Gamma({\mathrm{Ker}}(\rho))$}.$$
     The graded Jacobi identity implies that for any $e \in \Gamma(E_{-1}), f \in \Gamma(E_{-k}) $ (for $k \geq 2$ ) or $ f \in \Gamma({\mathrm{Ker}}(\rho)) $: 
      $$  l_1 (l_2(e,f)) =  l_2 (e,l_1(f)) .$$
       When applied to $e= \sigma(u)$, this reads $ [\nabla^\sigma_u , l_1 ] =0 $. The differential $l_1 $ is therefore preserved under parallel transportation along $L$, and, in particular, is of constant rank along the leaf $L$. 
    \item The first item implies that the cohomology of the complex \eqref{eq:leafwisecomplex} is a graded bundle over $L$. 
     \item[] We denote by $H_{\bullet}(\mathfrak i_L E) $ this cohomology, and call it the \textbf{graded cohomology of $(E,Q\equiv(\set{l_k}_{k\geq 1},\rho)) $ over the leaf $L$}.
      \item[]Notice that $H_{0}(\mathfrak i_L E) $ is the normal bundle $ {\mathfrak i}_L TM / TL $ of $L$ in $M$.
     \item Also, $A_L(E,Q):= \mathfrak i_L E_{-1} / l_1(\mathfrak i_L E_{-2})$ is a (transitive) Lie algebroid over $L$, when equipped with the anchor and Lie bracket:
      $$ \rho \colon \overline{e} \mapsto \rho(e) \hbox{ and } \set{\overline{e_1}, \overline{e_2}} = \overline{l_2(e_1,e_2)}    \hbox{ for all $e_1,e_2,e \in \Gamma( \mathfrak i_LE_{-1})$}.$$
      The horizontal bar stands for the natural map $\Gamma(\mathfrak i_LE_{-1}) \mapsto \Gamma( A_L(E,Q)) $.
     \item[] We call $A_L(E,Q)$ the {\textbf{holonomy Lie algebroid of the leaf $L$ for $(E,Q\equiv(\set{l_k}_{k\geq 1},\rho)) $}}.  Its isotropy Lie algebra bundle is 
     $H_{-1}(\mathfrak i_L E) $ by construction. 
    \item  Each one of the spaces $H_{\bullet}(\mathfrak i_L E)$ comes equipped with a canonical $A_L(E,Q)$-connection defined by:
    $$ \nabla^{}_{\overline{e}} \overline{f} = \overline{l_2(e,f)}  $$
    where $e \in \Gamma(\mathfrak i_L E_{-1}) $ and $ f \in \Gamma(\mathfrak i_L E_{-k})$, for $k \geq 2 $, or $f\in\Gamma(\mathfrak i_L{\mathrm{Ker}}(\rho))  $ and the horizontal bars are as before.
    \item[] As an exception to the previous rule, the $A_L(E,Q)$-connection on
    $H_{0}(\mathfrak i_L E) =  \mathfrak i_L TM / TL$ is defined by:
     $$ \nabla^{}_{\overline{e}} \overline{X} = \overline{[\rho(\tilde e), X]} $$
     for every $X \in \mathfrak X (M)$, and every section $\tilde e$ of $E_{-1}$ whose restriction to $L$ is $e$. Also, $X \mapsto \overline{X}$ stands for the natural map $\mathfrak X (M) \mapsto \Gamma( \mathfrak i_L TM / TL ) $.
     \item The higher Jacobi identities imply that the above $ A_L(E,Q)$-connections are flat. 
      \item[] The flatness of  the $ A_L(E,Q)$-connection on $ H_{-i}(\mathfrak i_L E) $ for $i \geq 1$ follows from the graded Jacobi identity, for all $e_1,e_2 \in \Gamma(E_{-1})$, $f \in \Gamma(E_{-k})$ with $k \geq 2$ or $f\in{\mathrm{Ker}}(\rho) $:
      $$ l_2(l_2(e_1,e_2), f )= l_2(e_1 , l_2(e_2,f) ) - l_2(e_2 , l_2(e_1,f) ) - l_3(e_1,e_2, l_1(f)) - l_1 \smalcirc l_3(e_1,e_2, f)$$
     since $ f \mapsto l_3(e_1,e_2, l_1(f)) + l_1 \smalcirc l_3(e_1,e_2, f)$ induce the zero map in cohomology.
     The flatness of the  $ A_L(E,Q)$-connection on the normal bundle is well-known and is easy to check: it uses the fact that the anchor map $\rho $ is a morphism (see, for instance, \cite{AZ}).
\end{enumerate}
Let us conclude this discussion with a Lemma:

\begin{lem}
Let $L$ be a leaf of a Lie $\infty$-algebroid $(E,Q\equiv(\set{l_k}_{k\geq 1},\rho)) $. 
The graded cohomology $ H_\bullet(\mathfrak i_L E ) $  over the leaf $L$ is a module over the holonomy Lie algebroid $A_L(E,Q)$ of $L$.
\end{lem}

Let us consider now the Berezinian bundle of the graded vector bundle $H_\bullet(\mathfrak i_L E) $: 
 $$ {\mathrm{Ber}} (H_\bullet(\mathfrak i_L E)):= \cdots \wedge^{\top} H_{-3}(\mathfrak i_L E) \otimes \wedge^{\top} (H_{-2}(\mathfrak i_L E))^* \otimes \wedge^{\top} H_{-1}(\mathfrak i_L E) \otimes \wedge^{\top} (\mathfrak i_L TM / TL)^* .  $$
By Lemma \ref{lem:berezinian}, there is a canonical isomorphism between the Berezinian  of a complex and the Berezinian of its cohomology.
In particular, there is a canonical isomorphism:
 $$ \mathfrak i_L : {\mathrm{Ber}} (E \oplus TM, l_1, -\rho )  \simeq  {\mathrm{Ber}} (H_\bullet(\mathfrak i_L E))
 .$$
 Recall that $(E\oplus TM,l_1,-\rho)$ stands for its linear part. Let us compute the \emph{modular class of the leaf $L$}, i.e. the characteristic class of this Lie algebroid representation. This Proposition shall be helpful.

\begin{prop}
\label{prop:bertober}
For every section $ e \in \Gamma(E_{-1})$ and every section $ \mu$ of the Berezinian bundle ${\mathrm{Ber}}(  E \oplus TM, l_1, - \rho)  $
$$ \nabla_{\overline{e}}^{Ber} \mathfrak i_L  \mu  =  \mathfrak i_L \ad_e^{Ber}\mu   .$$ 
In the previous equation, on the left hand side, $\nabla^{Ber}_{\overline{e}} $ is computed with the help of the $A_L(E,Q) $-connections on  $H_\bullet(\mathfrak i_L E) $.
\end{prop}
\begin{proof}
Let $ \overline{X_1} ,\dots, \overline{X_k} $ be a local trivialization of the normal bundle $\mathfrak i_L TM  $, arising from local vector fields $X_1, \dots,X_k$. For every $ i \geq 1$, let  $ e_{1}^{(i)}, \dots , e_{h_i}^{(i)}$ be sections of ${\mathrm{Ker}}(l_1) \subset \mathfrak i_L  E_{-i} $ (or ${\mathrm Ker}(\rho)$ for $ i=1$) whose classes $ \overline{e_{1}^{(i)}}, \dots , \overline{e_{h_i}^{(i)}}$ modulo $l_1(E_{-i-1}) $ form a local trivialization of the bundle $ H_{-i}(\mathfrak i_L E) $. 
There exists sections $ f_1^{(i)} \dots, f_{b_i}^{(i)} $ of $ \mathfrak i_L E_{-i}$ such that:
\begin{enumerate}
	\item local extensions of $ X_1 ,\dots, X_k ,  \rho(f_1^{(1)}) \dots, \rho( f_{b_1}^{(1)} )  $
	form a trivialization of $ TM $ (in a neighborhood of $L$). 
	\item $ l_1 (f_1^{(i+1)}) ,  \dots , l_1 (f_{b_{i+1}}^{(i+1)}) , e_{1}^{(i)}, \dots , e_{h_i}^{(i)} , f_1^{(i)} \dots, f_{b_i}^{(i)} $ form a trivialization of $\mathfrak i_L E_{-i} $.
	\end{enumerate}
	The isomorphism $\mathfrak i_L $, between Berezinian bundles, maps  the section $ \mu$ associated with the trivialization of $ E \oplus TM $  given above to the section associated to 
	the trivialization $   \overline{X_1} ,\dots, \overline{X_k} $ and $ \overline{e}_{1}^{(i)}, \dots , \overline{e}_{h_i}^{(i)} $. 
	Now, both actions of $E_{-1} $ on the Berezinian bundles are the alternate sums of the terms computed as in \eqref{eq:omegaii}, for all the elements of the trivializations above.
	In order to check that the Berezinian actions computed with respect to these two local trivializations coincide, it suffices to check that, for all possible indices, the term  in 
	$\ad_e^{Ber}\mu$ 
	where the bracket  $ l_2(e,f_j^{(i)}) $ appears adds up to zero with the term due to the bracket $ l_2(e,l_1(f_j^{(i)})) $
	(or $[\rho(e), \rho(f_j^{(1)})] $ for $i=1$). Also the term where the bracket of the form $l_2(e,e_j^{(i)})$ appears  is equal to the term where $\nabla_{\overline{e}} \overline{e}_j^{(i)}$ appears, and the term where  $ [\rho(e), X_i ]  $ appears is the term where $\nabla_{\overline{e}} \overline{X_i}$ appears.
 This completes the proof.
\end{proof}

A Lie $\infty $-algebroid morphism $ \Phi $ from a Lie $\infty $-algebroid
 $(E,Q\equiv(\set{l_k}_{k\geq 1},\rho)) $
 to
 $(E',Q'\equiv(\set{m_k}_{k\geq 1},\rho)) $
 induces in particular:
 \begin{enumerate}
     \item A chain map $\phi_1: E_\bullet \to E'_\bullet $
    \begin{equation} \label{eq:EtoE'}
    \xymatrix{ \cdots \ar[d]_{\phi_1} \ar[r]^{l_1}  &  E_{-2}\ar[d]_{\phi_1}  \ar[r]^{l_1} & \ar[d]_{\phi_1} E_{-1}   \ar[r]^{-\rho} &  TM \ar[d]^{=} \\ 
    \cdots  \ar[r]_{m_1}  & \ar[r]_{m_1} E_{-2}'& E_{-1}'   \ar[r]_{-\rho'} &  TM \ar[u] }
    \end{equation}
     \item and a degree zero vector bundle morphism $\phi_2: S^2(E) \to E $  such that 
      for every sections $e_1,e_2 \in \Gamma(E)$:
       \begin{eqnarray} \label{eq:EtoE'2} \lefteqn{\phi_1(\{e_1,e_2\})}\nonumber\\ &=&\{\phi_1(e_1),\phi_1(e_2)\}+ l_1 \phi_2(e_1,e_2)  - \phi_2(l_1(e_1),e_2) - (-1)^{|e_1|}\phi_2(e_1,l_1 (e_2)) . 
       \end{eqnarray}
 \end{enumerate}
Let us assume that both $(E,Q\equiv(\set{l_k}_{k\geq 1},\rho)) $
 and
 $(E',Q'\equiv(\set{m_k}_{k\geq 1},\rho')) $
 have the same leaves, and let $L$ be such a leaf.
 The Lie $ \infty$-morphism above then induces, by Equations (\ref{eq:EtoE'} - \ref{eq:EtoE'2})   \begin{enumerate}
        \item a Lie algebroid morphism $A_L(E,Q)  \longrightarrow A_L (E',Q')$,
        \item a graded vector bundle morphism:
         $$H_\bullet (\mathfrak i_L E) \longrightarrow H_\bullet (\mathfrak i_L E') $$ 
         
         \item These morphisms intertwines the respective actions of  
         $A_L(E,Q)$ and $A_L (E',Q')$ on $ H_\bullet (\mathfrak i_L E)$ and $ H_\bullet (\mathfrak i_L E')$.
    \end{enumerate}
In particular, homotopy equivalent  Lie $\infty $-algebroids $(E,Q\equiv(\set{l_k}_{k\geq 1},\rho)) $
 and
 $(E',Q'\equiv(\set{m_k}_{k\geq 1},\rho')) $ have the same leaves, and a homotopy equivalence induces 
    \begin{enumerate}
        \item a Lie algebroid isomorphism $A_L(E,Q)  \simeq A_L (E',Q')$,
        \item a graded vector bundle isomorphism:
         $$H_\bullet (\mathfrak i_L E) \simeq H_\bullet (\mathfrak i_L E') $$
         \item which intertwines the respective actions of  
         $A_L(E,Q)$ and $A_L (E',Q')$ on $ H_\bullet (\mathfrak i_L E)$ and $ H_\bullet (\mathfrak i_L E')$.
    \end{enumerate}
 By Lemma \ref{lem:berezinian}, a homotopy equivalence between two complexes induces a canonical isomorphism of their Berezinian bundles, and the following diagram commutes.
         $$ \xymatrix{ {\mathrm{Ber}}(E \oplus TM, l_1, -\rho ) \ar[dr]^{\mathfrak i_L} \ar[dd]^{\simeq} &  \\ &  \mathrm{Ber}(H_\bullet(i_L E))\simeq \mathrm{Ber}(H_\bullet(i_L E'))\\ \ar[uu] {\mathrm{Ber}}(  E' \oplus TM, l_1', -\rho'  ) \ar[ur]_{\mathfrak i_L}& } 
         $$ 
The following corollary of Proposition \ref{prop:bertober} follows from the previous discussion:

\begin{cor}
A homotopy equivalence between two Lie $\infty$-algebroids intertwines the modular classes of their leaves. 
In particular, unimodularity of a given leaf is preserved under homotopy equivalence.
\end{cor}

\end{document}